\documentclass[preprint,12pt]{elsarticle}

\usepackage{lineno,hyperref}
\modulolinenumbers[5]

\usepackage{ifpdf}
\usepackage{bm}
\usepackage{mathdots}
\usepackage{soul, color}
\usepackage{amsmath}
\usepackage{mathrsfs}
\usepackage{subfig}
\usepackage{graphicx}
\usepackage{epstopdf}
\usepackage{epsfig}
\usepackage{yfonts}
\usepackage{xcolor}
\usepackage[abs]{overpic}
\usepackage{enumitem}

\newtheorem{theorem}{Theorem}[section]
\newtheorem{lemma}[theorem]{Lemma}
\newtheorem{proposition}[theorem]{Proposition}
\newtheorem{corollary}[theorem]{Corollary}

\newtheorem{definition}{Definition}[section]

\newenvironment{proof}[1][Proof]{\begin{trivlist}
\item[\hskip \labelsep {\bfseries #1}]}{\end{trivlist}}

\usepackage[top=1in, bottom=1in, left=1in, right=1in]{geometry}

\newcommand{\Wnp}{\mathcal{W}^{n,p}}

\newcommand{\Wni}{\mathcal{W}^{n,\infty}}

\newcommand{\R}{\ifmmode\mathbb{R}\else$\mathbb{R}$\fi}
\newcommand{\N}{\ifmmode\mathbb{N}\else$\mathbb{N}$\fi}
\newcommand{\Q}{\ifmmode\mathbb{Q}\else$\mathbb{Q}$\fi}
\newcommand{\Z}{\ifmmode\mathbb{Z}\else$\mathbb{Z}$\fi}

\newcommand{\tn}[1]{\textnormal{#1}}
\newcommand{\myto}{\mathop{\raisebox{-0pt}{\scalebox{2.6}[1]{$\longrightarrow$}}}}

\newcommand{\NNinput}{\tn{\#input}}
\newcommand{\NNwidth}{\, \tn{widthvec}}

\newcommand{\xal}{{\bm{\alpha}}}
\newcommand{\xx}{{\bm{x}}}

\renewcommand{\epsilon}{\varepsilon}
\renewcommand{\subset}{\subseteq}
\renewcommand{\frac}[2]{\tfrac{#1}{#2}}

\usepackage{amssymb}

\journal{Journal}

\begin{document}

\begin{frontmatter}



\title{Simultaneous Neural Network {Approximation for Smooth Functions}}

\date{}

\author{Sean Hon\corref{cor1}}
\address{Department of Mathematics, Hong Kong Baptist University, Kowloon Tong, Hong Kong SAR}
\ead{seanyshon@hkbu.edu.hk}

\author{Haizhao Yang\corref{cor2}}
\address{Department of Mathematics, Purdue University, IN 47907, USA}
\ead{haizhao@purdue.edu}

\begin{abstract}
We establish in this work approximation results of deep neural networks for smooth functions measured in Sobolev norms, motivated by recent development of numerical solvers for partial differential equations using deep neural networks. {Our approximation results are nonasymptotic in the sense that the error bounds are explicitly characterized in terms of both the width and depth of the networks simultaneously with all involved constants explicitly determined.} Namely, for $f\in C^s([0,1]^d)$, we show that deep ReLU networks of width $\mathcal{O}(N\log{N})$ and of depth $\mathcal{O}(L\log{L})$ can achieve a nonasymptotic approximation rate of $\mathcal{O}(N^{-2(s-1)/d}L^{-2(s-1)/d})$ with respect to the $\mathcal{W}^{1,p}([0,1]^d)$ norm for $p\in[1,\infty)$. If either the ReLU function or its square is applied as activation functions to construct deep neural networks of width $\mathcal{O}(N\log{N})$ and of depth $\mathcal{O}(L\log{L})$ to approximate $f\in C^s([0,1]^d)$, the approximation rate is $\mathcal{O}(N^{-2(s-n)/d}L^{-2(s-n)/d})$ with respect to the $\mathcal{W}^{n,p}([0,1]^d)$ norm for $p\in[1,\infty)$. 
\end{abstract}


\end{frontmatter}



\section{Introduction}

Over the past decades, deep neural networks have made remarkable impacts in various areas of science and engineering. With the aid of high-performance computing equipment and abundance of high quality data, neural network based methods outperform traditional machine learning methods in a wide range of applications, including image classification, object detection, speech recognition, to name just a few. The success of neural networks also motivates its applications in scientific computing including the recovery of governing equations for mathematical modeling and prediction \cite{PhysRevLett.120.143001,E_2017,QIN2019620,LONG2019108925,HARLIM2021109922,CiCP-25-947} and solving partial differential equations {(PDEs)} \cite{RAISSI2019686,E_han_Jentzen_2017,Han8505,gitta_petersen_raslan_2021,Deenis_pilpp_Jentzen_2021,geist_petersen_raslan_2021,khoo_lu_ying_2021,GU2021110444,2020arXiv201208023C,2021arXiv210606682L}.

Neural network based methods have evoked many open problems in mathematical theory, notwithstanding their success in practice. In a typical supervised learning algorithm, a potential high-dimensional target function $f(x)$ defined on a domain $\Omega$ is to be learned from a finite set of data samples $\{(\mathbf{x}_i,f(\mathbf{x}_i))\}_{i=1}^n$. When a deep network is integrated in the learning process, one needs to identify a deep network $\phi(\mathbf{x};\mathbf{\theta_{S}})$ with $\mathbf{\theta_{S}}$ being the hyperparameter to determine $f(\mathbf{x})$ for unseen data samples $\mathbf{x}$, namely the following optimization problem arises

\begin{equation}\label{eqn:theta_S}
\mathbf{\theta_{S}} = \text{arg} \min_{\mathbf{\theta}}R_S(\mathbf{\theta}):= \text{arg} \min_{\mathbf{\theta}}  \frac{1}{n} \sum_{i=1}^{n}l ( \phi (x_i ; \theta), f(x_i) )
\end{equation}
where $l$ denotes a loss function.

Now, let us inspect the overall inference error which is estimated by $R_D (\theta_{S})$, where
$
R_D (\theta) := E_{\mathbf{x}\sim U(\Omega)}[l ( \phi (\mathbf{x} ; \mathbf{\theta}), f(\mathbf{x}) )]
$. In reality, $U(\Omega)$ is unknown and only finitely many samples from this distribution are available. Hence, the empirical loss $R_S(\mathbf{\theta})$ is minimized hoping to obtain $\phi(\mathbf{x};\mathbf{\theta_{S}})$, instead of minimizing the population loss $R_D(\mathbf{\theta})$ to obtain $\phi(\mathbf{x};\mathbf{\theta_{D}})$, where $\mathbf{\theta_{D}}=\text{arg}\min_{\mathbf{\theta}} R_D (\theta)$. In practice, a numerical optimization method to solve (\ref{eqn:theta_S}) may result in a numerical solution (denoted as $\mathbf{\theta_{{\mathcal{N}}}}$) that may not be a global minimizer $\mathbf{\theta_{{S}}}$. Therefore, the actually learned neural network to infer $f(x)$ is $\phi(\mathbf{x};\mathbf{\theta_{{\mathcal{N}}}})$ and the corresponding inference error is measured by $R_D(\mathbf{\theta_{\mathcal{N}}})$. By the discussion just above, it is crucial to quantify $R_D(\mathbf{\theta_{\mathcal{N}}})$ to see how good the learned neural network $\phi(\mathbf{x};\mathbf{\theta_{{\mathcal{N}}}})$ is, since $R_D(\mathbf{\theta_{\mathcal{N}}})$ is the expected inference error overall possible data samples. Note that
\begin{eqnarray}
R_D(\mathbf{\theta_{\mathcal{N}}}) 
\leq  \underbrace{R_D(\mathbf{\theta_D})}_{\rm{approximation}} + \underbrace{[R_S(\mathbf{\theta_{\mathcal{N}}}) - R_S(\mathbf{\theta_S})]}_{\rm{optimization}}+\underbrace{[R_D(\mathbf{\theta_{\mathcal{N}}}) - R_S(\mathbf{\theta_{\mathcal{N}}})]+ [R_S(\mathbf{\theta_D}) - R_D(\mathbf{\theta_D})]}_{\rm{generalization}},\label{eqn:learning_error_main}
\end{eqnarray}
where the inequality comes from the fact that $[R_S(\mathbf{\theta_S}))- R_S(\mathbf{\theta_D})] \leq 0$ since $\mathbf{\theta_S}$ is a global minimizer of $R_S(\mathbf{\theta})$. 

Synergies from approximation theory \cite{Yarotsky2017,Schwab2019,Bolcskei2019,Shen2019,Montanelli2020,Montanelli2019,lu2020deep,Remi_Gitta_Morten_V_2021,Petersen2018, E2018,Weinan2019,E_MA_WU_2021,SIEGEL2020313,2018arXiv180703973H,devore_hanin_petrova_2021,Nonlinear_app_2021}, optimization theory \cite{Jacot2018,MeiE7665,pmlr-v99-mei19a,pmlr-v97-arora19a,2018arXiv181002054D,pmlr-v70-zhong17a,pmlr-v97-allen-zhu19a,pmlr-v97-du19c,8409482}, and generalization theory \cite{Jacot2018,Cao2019,Weinan2019,Berner2020,luo2020twolayer,2021arXiv210101708L,2021arXiv210501228L} have led to many recent advances in the mathematical investigation for deep learning, regarding the error bound (\ref{eqn:learning_error_main}). All of these areas, having different emphases and {angles}, have fostered many separate research directions. 

Providing an estimate of the first error term of (\ref{eqn:learning_error_main}), $R_D(\mathbf{\theta_D})$, belongs to the regime of approximation theory and is of the main concerns in this paper.   
{The results established in this work provide an upper bound of $R_D(\mathbf{\theta_D})$, estimated explicitly in terms of the size of the network, e.g. its width and depth, with a nonasymptotic approximation rate for smooth functions measured in the Sobolev norms. {Our approximation results are nonasymptotic in the sense that the all involved constants are explicitly determined.} We summarize our first main result in the following {on the networks with rectified linear unit (ReLU) $\sigma_1$ as the activation function}:}

{
\begin{theorem}\label{coro:function_sigma_1}
Suppose that $f \in C^s([0,1]^d)$ with $s>1\in\mathbb{N}^+$ satisfies $\| \partial^{\bm{\alpha}} f  \|_{L^{\infty}([0,1]^d)} < 1$ for any $\bm{\alpha} \in \mathbb{N}^d$ with $|\xal| \leq s$. For any $N,L \in \mathbb{N}^+$ and $p\in[1,\infty)$, there exists a $\sigma_1$-$\tn{NN}$ $\phi$ with width $16s^{d+1}d(N+2)\log_2{( {8}N)}$ and depth $27s^2 (L+2)\log_2(4L)$ such that
\[
\| f - \phi \|_{\mathcal{W}^{1,p}([0,1]^d)} \leq   {85}(s+1)^d 8^{s}N^{-2(s-1)/d}L^{-2(s-1)/d}.
\]
\end{theorem}
}

{For smooth functions, we show in Theorem \ref{coro:function_sigma_1} that deep ReLU networks of width $\mathcal{O}(N\log{N})$ and of depth $\mathcal{O}(L\log{L})$ can achieve an approximation rate of $\mathcal{O}(N^{-2(s-1)/d}L^{-2(s-1)/d})$ with respect to the $\mathcal{W}^{1,p}([0,1]^d)$ norm. Note that ReLU networks are in fact piecewise linear functions, so they have at most nonzero first (weak) derivatives. The above approximation rate in the $L^{p}$ norm estimated in terms of $N$ and $L$ was already covered in \cite{lu2020deep}. When the $\mathcal{W}^{n,p}$ norm is considered, where $0<n<1$ is not an integer, the interpolation technique used in \cite{Guhring2020} can be combined together with our method here to develop new approximation rates, which is left as future work.}

{
To achieve rates truly in terms of width and depth without the logarithm terms, we also have the following corollary by setting $\widetilde{N}=\mathcal{O}(N\log N)$ and $\widetilde{L}=\mathcal{O}(L\log L)$ and making use of
\[
(N\ln N)^{-2(s-1)/d}(L\ln L)^{-2(s-1)/d} \leq \mathcal{O}(N^{-2(s-\rho)/d}L^{-2(s-\rho)/d} )
\]
for $\rho \in (1,s)$.
\begin{corollary}\label{coro:function_sigma_1_linear}
Suppose that $f \in C^s([0,1]^d)$ with $s>1\in\mathbb{N}^+$ satisfies $\| \partial^{\bm{\alpha}} f  \|_{L^{\infty}([0,1]^d)} < 1$ for any $\bm{\alpha} \in \mathbb{N}^d$ with $|\xal| \leq s$. For any $N,L \in \mathbb{N}^+$, $\rho \in (1,s)$, and $p\in[1,\infty)$, there exist $C_1(s,d)$, $C_2(s,d)$, $C_3(s,d,\rho)$, and a $\sigma_1$-$\tn{NN}$ $\phi$ with width $C_1N$ and depth $C_1L$ such that
\[
\| f - \phi \|_{\mathcal{W}^{1,p}([0,1]^d)} \leq  C_3 N^{-2(s-\rho)/d}L^{-2(s-\rho)/d}.
\]
\end{corollary}
Note that the constants $C_1,C_2$ and $C_3$ in Corollary \ref{coro:function_sigma_1_linear} can be explicitly determined and we leave it to readers.}

{To obtain approximation results in terms of the number of parameters in neural networks, the following corollary is followed by setting $N=\mathcal{O}(1)$ and $\epsilon=\mathcal{O}(L^{-2(s-1)/d})$ from Theorem \ref{coro:function_sigma_1}.
\begin{corollary}\label{coro:function_sigma_1_assym}
Suppose that $f \in C^s([0,1]^d)$ with $s>1\in\mathbb{N}^+$ satisfies $\| \partial^{\bm{\alpha}} f  \|_{L^{\infty}([0,1]^d)} < 1$ for any $\bm{\alpha} \in \mathbb{N}^d$ with $|\xal| \leq s$. Given any $\epsilon>0$ {and $p\in[1,\infty)$}, there exists a $\sigma_1$-$\tn{NN}$ $\phi$ with $\mathcal{O}(\epsilon^{-d/(2(s-1))}\ln \frac{1}{\epsilon})$ parameters such that
\[
\| f - \phi \|_{\mathcal{W}^{1,p}([0,1]^d)} \leq  \epsilon.
\]
\end{corollary}
Corollary \ref{coro:function_sigma_1_assym} partially recovers \cite[Corollary 4.1]{Guhring2020} in which the parameters are of order $\mathcal{O}(\epsilon^{-d/(s-1)}\ln^2 ( {\epsilon}^{-s/(s-1)})  )$.
}

Considering a smoother neural network in which either the ReLU function {$\sigma_1$} or its square {$\sigma_1^2$} is applied, we provide similar results in the following Theorem \ref{corol:function_sigma_2}. Namely, these networks of width $\mathcal{O}(N\log{N})$ and of depth $\mathcal{O}(L\log{L})$ can achieve the approximation rate of $\mathcal{O}(N^{-2(s-n)/d}L^{-2(s-n)/d})$ with respect to the $\mathcal{W}^{n,p}([0,1]^d)$ norm. It is worth noting that the confinement of space is now relaxed to $\mathcal{W}^{n,p}([0,1]^d)$ from $\mathcal{W}^{1,p}([0,1]^d)$ with such smoother networks.

\begin{theorem}\label{corol:function_sigma_2}
Suppose that $f \in C^s([0,1]^d)$ with $s\in \mathbb{N}^+$ satisfies $\| \partial^{\bm{\alpha}} f  \|_{L^{\infty}([0,1]^d)} < 1$ for any $\bm{\alpha} \in \mathbb{N}^d$ with $|\xal| \leq s$. For any $N,L \in \N^+$ satisfying $(L-2-\log_2 N)N \geq s$, and $p \in [1,\infty)$, there exists a $\sigma_2$-$\tn{NN}$ $\phi$ with width $16s^{d+1}d(N+2)\log_2{({8}N)}$ and depth $10(L+2)\log_2(4L)$ such that
\[
\| f - \phi \|_{\mathcal{W}^{n,p}([0,1]^d)} \leq  {3}(s+1)^d 8^{s-n}N^{-2(s-n)/d}L^{-2(s-n)/d},
\]
where $n < s$ is a positive integer.
\end{theorem}

Our work developed here focuses particularly on the Sobolev norms which are suitable for studying partial differential equations. {The above results concern the ReLU$^k$ network approximation of functions with respect to Sobolev norms. Along this research direction, the work in \cite{Guhring2020} provides asymptotic upper bounds with unknown parameters with respect to $\mathcal{W}^{n,p}$ norm, where $n$ is a fraction between zero and one, based on the means of localized polynomials. Making use of the same polynomials, asymptotic approximation results measured with respect to high order $\mathcal{W}^{n,p}$ norm were provided in \cite{2020arXiv200616822G}. These results were developed for neural networks with smoother activation functions, in addition to ReLU networks. Measured in $\mathcal{W}^{1,p}$ norm, it was shown in \cite{doi:10.1142/S0219530519410136} that ReLU networks can achieve essentially the same approximation rates as free-knot spline approximations. In contrast with these existing results, our approximation results are nonasymptotic with known pre-factors, established in the spirit of explicit error characterization methodology developed in \cite{lu2020deep,shen2020deep}. }

Recall from \cite[Theorem 4.2]{DeVore_Howad_micchelli_1989} by DeVore et al., we have the following ostensible negative result on the continuity of the weight selection. Suppose that there is a continuous map $\Sigma$ from the unit ball of Sobolev space with smoothness $n$, i.e. $\Wnp$, to $\mathbb{R}^{W}$ such that $ \|f - g(\Sigma(f)) \|_{L^p} \leq \epsilon$ for all $f \in \Wnp$, where $W$ denotes a fixed number of parameters and $g$ is a map realizing a deep neural network from a given set of parameters in $\mathbb{R}^{W}$ to the unit ball in $\Wnp$, then $W \geq C \epsilon^{-n/d}$ with some constant $C$ depending only on $n$. This in a way means any such constructive approximation of ReLU networks cannot have a continuous weight selection property if the approximation rate is better than $C \epsilon^{-n/d}$, and hence a stable numerical implementation with such an error rate does not exist. It must, however, note that \cite[Theorem 4.2]{DeVore_Howad_micchelli_1989} is basically a min-max criterion for evaluating continuous weight selection maps, describing the worst case. That is, the approximation result is obtained by minimizing over all continuous maps $\Sigma$ and network realizations $g$ and maximizing over all target functions. For most smooth functions practically encountered in applications, the theorem does not eliminate the possible cases in which they might still enjoy a continuous weight selection. Thus, there could be a stable numerical algorithm that can achieve our derived approximation rate. {In other words, there is a special subclass of functions which arise in practice for which a continuous assignment of the weights exists. It is interesting future work to characterize such a subclass. Finally, to the best of our knowledge, there is not any efficient numerical algorithm} to achieve the approximation rate with continuous weight selection especially when the dimension is large. Therefore, approximation results with or without continuous weight selection both have the difficulty of numerical implementations. Designing numerical algorithms to achieve these rates has been an active research field recently.

{It is remarked that providing error estimate in the context of high-dimensional PDE problem using the newly introduced Floor-ReLU networks \cite{shen2020deep}, which are fully connected neural networks with either Floor $\lfloor x \rfloor$ or ReLU activation function in each neuron, will be an intriguing option for future work, since these networks conquer the curse of dimensionality in terms of approximation theory.} Other neural networks can also be considered. Along this line of work, the novel Floor-Exponential-Step networks with merely three hidden layers and width of order $\mathcal{O}(N)$ constructed in \cite{SHEN2021160} can approximate $d$-dimensional Lipschitz continuous functions with an exponentially small error rate of order $\mathcal{O}(\sqrt{d}2^{-N})$. In \cite{2021arXiv210702397S}, neural networks with a simple and computable continuous activation function and a fixed finite number of neurons were developed, which achieve the universal approximation property for all high-dimensional continuous functions.

This paper is {organized} as follows. In Section \ref{section:pre}, Sobolev spaces are briefly introduced and a number of useful existing results on ReLU network approximation are given. We provide our main approximation results for ReLU networks in Section \ref{section:main}. {Several auxiliary results concerning approximating high order polynomials using ReLU networks are first provided in Section \ref{section:ReLU}, and our main results on ReLU network are given in Section \ref{subsect:proof_thm_sigma_1}}. In addition to ReLU networks, we show similar results for the smoother neural networks with ReLU square in Section \ref{section:signma_2}. 

\section{Preliminaries}\label{section:pre}

We provide in this section some useful preliminaries of notations and basic approximation results.

\subsection{Deep Neural Networks}

Let us summarize all basic notations used in deep neural networks as follows.
\begin{enumerate}

     \item Matrices are denoted by bold uppercase letters. For instance,  $\bm{A}\in\mathbb{R}^{m\times n}$ is a real matrix of size $m\times n$, and $\bm{A}^T$ denotes the transpose of $\bm{A}$. 
     
     \item Vectors are denoted as bold lowercase letters. For example, $\bm{v}\in \R^n$ is a column vector of size $n$. Correspondingly, $\bm{v}(i)$ is the $i$-th element of $\bm{v}$. $\bm{v}=[v_1,\cdots,v_n]^T=\left[\hspace{-4pt}\begin{array}{c}
    \vspace{-3pt} v_1 \\ \vspace{-5pt} \vdots \\ v_n
     \end{array}\hspace{-4pt}\right]$ are vectors consisting of numbers $\{v_i\}$ with $\bm{v}(i)=v_i$.
     
      \item 
    A $d$-dimensional multi-index is a $d$-tuple
    $\xal=[\alpha_1,\alpha_2,\cdots,\alpha_d]^T\in \N^d.$
    Several related notations are listed below.
    \begin{enumerate}
        \item  $|\xal|=|\alpha_1|+|\alpha_2|+\cdots+|\alpha_d|$;
        \item $\xx^\xal=x_1^{\alpha_1}  x_2^{\alpha_2} \cdots x_d^{\alpha_d}$, where $\xx=[x_1,x_2,\cdots,x_d]^T$;
        \item $\xal!=\alpha_1!\alpha_2!\cdots \alpha_d!$;
    \end{enumerate}

    \item Let $B_{r,|\cdot|}(\xx)\subseteq \R^d$ be the closed ball with a center $\xx\subseteq \R^d$ and a radius $r$ measured by the Euclidean distance. Similarly, $B_{r,\|\cdot\|_{\ell^\infty}}(\xx)\subseteq \R^d$ is a ball measured by the discrete $\ell^\infty$-norm of a vector.
          
     \item Assume $\bm{n}\in \N^n$, then $f(\bm{n})=\mathcal{O}(g(\bm{n}))$ means that there exists positive $C$ independent of $\bm{n}$, $f$, and $g$ such that $ f(\bm{n})\le Cg(\bm{n})$ when all entries of $\bm{n}$ go to $+\infty$.
     
     \item {We use $\sigma$ to denote an activation function}. Let $\sigma_1:\R\to \R$ denote the rectified linear unit (ReLU), i.e. $\sigma_1(x)=\max\{0,x\}$. With the abuse of notations, we define $\sigma_1:\R^d\to \R^d$ as $\sigma_1(\xx)=\left[\begin{array}{c}
          \max\{0,x_1\}  \\
          \vdots \\
          \max\{0,x_d\}
     \end{array}\right]$ for any $\xx=[x_1,\cdots,x_d]^T\in \R^d$.
     
     \item {Furthermore, we let the activation function $\sigma_2:\R\to \R$ be either $\sigma_1$ or $\sigma_1^2$. Similar to $\sigma_1$, we define the action of $\sigma_2$ on a vector $\xx$.}

     \item We will use $\tn{NN}$ as a neural network for short and \textbf{$\sigma_r$-$\tn{NN}$ to specify an $\tn{NN}$ with activation functions $\sigma_t$ with $t\leq r$}. We will also use Python-type notations to specify a class of $\tn{NN}$'s, e.g., $\sigma_1$-$\tn{NN}(\tn{c}_1;\ \tn{c}_2;\ \cdots;\ \tn{c}_m)$ is a set of ReLU FNNs satisfying $m$ conditions given by $\{\tn{c}_i\}_{1\leq i\leq m}$, each of which may specify the number of inputs ($\NNinput$), the total number of nodes in all hidden layers ($\#$node), the number of hidden layers ($\#$layer), the number of total parameters ($\#$parameter), and the width in each hidden layer (widthvec), the maximum width of all hidden layers (maxwidth), etc. For example, if $\phi\in \sigma_1$-$\tn{NN}(\NNinput=2; \NNwidth=[100,100])$,  then $\phi$  satisfies
     \begin{enumerate}
         \item $\phi$ maps from $\R^2$ to $\R$.
         \item $\phi$ has two hidden layers and the number of nodes in each hidden layer is $100$.
     \end{enumerate}
     \item $[n]^L$ is short for $[n,n,\cdots,n]\in \N^L$. 
     For example, \[\tn{NN}(\NNinput=d;\NNwidth=[100,100])=\tn{NN}(\NNinput=d;\NNwidth=[100]^2).\]
     
     \item For $\phi\in \sigma$-$ \tn{NN}(\NNinput=d;\NNwidth=[N_1,N_2,\cdots,N_L])$, if we define $N_0=d$ and $N_{L+1}=1$, then the architecture of $\phi$ can be briefly described as follows:
    \begin{equation*}
    \begin{aligned}
    \bm{x}=\tilde{\bm{h}}_0 \myto^{\bm{W}_1,\ \bm{b}_1} \bm{h}_1\mathop{\longrightarrow}^{\sigma} \tilde{\bm{h}}_1 \cdots \myto^{\bm{W}_L,\ \bm{b}_L} \bm{h}_L\mathop{\longrightarrow}^{\sigma} \tilde{\bm{h}}_L \mathop{\myto}^{\bm{W}_{L+1},\ \bm{b}_{L+1}} \phi(\bm{x})=\bm{h}_{L+1},
    \end{aligned}
    \end{equation*}
    where $\bm{W}_i\in \R^{N_{i}\times N_{i-1}}$ and $\bm{b}_i\in \R^{N_i}$ are the weight matrix and the bias vector in the $i$-th linear transform in $\phi$, respectively, i.e., \[\bm{h}_i :=\bm{W}_i \tilde{\bm{h}}_{i-1} + \bm{b}_i,\quad \tn{for $i=1$, $\dots$, $L+1$,}\]  and
    \[
       \tilde{\bm{h}}_i=\sigma(\bm{h}_i),\quad \tn{for $i=1$, $\dots$, $L$.}
    \]
    $L$ in this paper is also called the number of hidden layers in the literature.

     \item The expression, an FNN with width $N$ and depth $L$, means
     \begin{enumerate}
         \item The maximum width of this FNN for all hidden layers less than or equal to $N$.
         \item The number of hidden layers of this FNN less than or equal to $L$.
     \end{enumerate}
\end{enumerate}

\subsection{Sobolev Spaces}

We will use $D$ to denote the weak derivative of a single variable function and $D^\xal$ to denote the partial derivative $D^{\alpha_1}_1 D^{\alpha_2}_2 \dots D^{\alpha_d}_d$ of a $d$-dimensional function with $\alpha_i$ as the order of derivative $D_i$ in the $i$-th variable and $\xal=[\alpha_1,\dots,\alpha_d]^T$. Let $\Omega$ denote an open subset of $\mathbb{R}^d$ and $L^p(\Omega)$ be the standard Lebesgue space on $\Omega$ for $p\in[1,\infty]$. We write $\nabla f:=[D_1 f, \dots, D_d f]^T$. $\partial \Omega$ is the boundary of $\Omega$. Let $\mu(\cdot)$ be the Lebesgue measure.  For $f(x)\in\Wnp(\Omega)$, we use the notation
\[
\|f\|_{\Wnp(\Omega)} = \|f\|_{\Wnp} = \|f(x)\|_{\Wnp(\Omega,\mu)},
\]
if the domain is clear from the context and we use the Lebesgue measure. 

\begin{definition} (Sobolev Space) Let $n\in \mathbb{N}_0$ and $1\leq p\leq \infty$. Then we define the Sobolev space 
\[
\Wnp(\Omega):=\{ f\in L^p(\Omega):D^\xal f \in L^p(\Omega)\text{ for all }\xal\in \mathbb{N}_0^d \text{ with }|\xal|\leq n \}
\]
with a norm
\[
\|f\|_{\Wnp(\Omega)}:=\Bigg( \sum_{0\leq |\xal| \leq n} \|D^\xal f\|^p_{L^p(\Omega)} \Bigg)^{1/p},
\]
if $p<\infty$, and 
\[
\|f\|_{\Wni(\Omega)} := \max_{0\leq |\xal|\leq n} \| D^\xal f\|_{L^\infty(\Omega)}.
\]
\end{definition}

\subsection{Auxiliary Neutral Network Approximation Results}

We first give the following useful lemmas on several ReLU networks approximation results with explicit error characterization measured in the $L^{\infty}$ norm for polynomials.

\begin{lemma}\label{lem:NNex1}
The followings lemmas are satisfied by $\sigma_1$-$\tn{NN}$s.
\begin{enumerate}[label=(\roman*)]
\item Any one-dimensional continuous piecewise linear function with $N$ breakpoints can be exactly realized by a one-dimensional $\sigma_1$-$\tn{NN}$ with one-hidden layer and $N$ neurons.
\item Any identity map in $\R^d$ can be exactly realized by a d-dimensional $\sigma_1$-$\tn{NN}$ with one hidden layer and $2d$ neurons.
\item (\cite[Lemma 5.1]{lu2020deep}) For any $N,L\in \N^+$, there exists a $\sigma_1$-$\tn{NN}$ $\phi$ with width $3N$ and depth $L$ such that
    \begin{equation*}
    \|\phi(x)-x^2\|_{L^{\infty}([0,1])} \le N^{-L}.
    \end{equation*}
\item (\cite[Lemma 4.2]{lu2020deep}) For any $N,L\in \N^+$ and $a,b\in \R$ with $a<b$, there exists a $\sigma_1$-$\tn{NN}$ $\phi$ with width $9N+1$ and depth $L$ such that
    \begin{equation*}
    \|\phi(x,y)-xy \|_{L^{\infty}([a,b]^2)}  \le 6(b-a)^2N^{-L}.
    \end{equation*}
    \item (\cite[Theorem 4.1]{lu2020deep}) Assume $P(\xx)=\xx^\xal=x_1^{\alpha_1}x_2^{\alpha_2}\cdots x_d^{\alpha_d}$ for $\xal\in \N^d$ with $\|\xal\|_1 \leq k \in \N^+$. For any $N,L\in \N^+$, there exists a $\sigma_1$-$\tn{NN}$ $\phi$ with width $9(N+1)+k-1$ and depth $7k^2L$ such that
    \begin{equation*}
    \|\phi(\xx)-P(\xx)\|_{L^{\infty}([0,1]^d)}\le 9k(N+1)^{-7kL}.
    \end{equation*}
  
\end{enumerate}
\end{lemma}

The following two propositions will also be used in proving our main results.

\begin{proposition}(Step function approximations \cite[Proposition 4.3]{lu2020deep})\label{prop:step_sigma_1}
For any $N,L,d \in \mathbb{N}^+$ and $\delta \in (0,\frac{1}{3K}]$ with $K = \lfloor {N^{1/d}} \rfloor^2 \lfloor L^{2/d} \rfloor$, there exists a one-dimensional $\sigma_1$-$\tn{NN}$ $\phi$ with width $4\lfloor N^{1/d} \rfloor+3$ and depth $4L+5$ such that
\[
\phi(x) = k,\quad \text{if} ~ x\in \big[\frac{k}{K},\frac{k+1}{K}- \delta \cdot 1_{\{ k<K-1 \}}\big]
\]
for $k=0,1,\dots,K-1$.
\end{proposition}

\begin{proposition}(Point matching \cite[Proposition 4.4]{lu2020deep})\label{prop:point_sigma_1}
Given any $N,L,s \in \mathbb{N}^+$ and $\xi_i \in [0,1]$ for $i=0,1,\dots,N^2L^2-1$, there exists a $\sigma_1$-$\tn{NN}$ $\phi$ with width $16s(N+1)\log_2(8N)$ and depth $(5L+2)\log_2{(4L)}$ such that
\begin{enumerate}
\item $|\phi(i)-\xi_i| \leq N^{-2s}L^{-2s},$
for $i=0,1,\cdots,N^2L^2-1$;
\item $0\leq \phi(x) \leq 1,$
$x\in \mathbb{R}$
\end{enumerate}
\end{proposition}

\section{{Proof of Our Main Results}}\label{section:main}

In this section, the proofs of our main results are given. We first define the following notion of subset of $[0,1]^d$, before presenting our main approximation results. Given any $K\in N^{+}$ and $\delta \in (0,\frac{1}{K})$, define a trifling region {$\Omega{([0,1]^d,K,\delta)}$} of $[0,1]^d$ as
    \begin{equation}\label{eqn:trifling_region}
    \Omega{([0,1]^d,K,\delta)}:= \cup_{i=1}^d \Big\{\xx =[x_1,x_2,\cdots,x_d]^T \in [0,1]^d: x_i\in \cup_{k=1}^{K-1}(\frac{k}{K}-\delta, \frac{k}{K}) \Big\}. 
    \end{equation}

    \subsection{Approximations Using ReLU Neural Networks}\label{section:ReLU}

{We first present the following Theorem \ref{thm:function_sigma_1}, which concerns the approximation result by ReLU networks outside the trifling region {$\Omega{([0,1]^d,K,\delta)}$}.}

\begin{theorem}\label{thm:function_sigma_1}
Suppose that $f \in C^s([0,1]^d)$ with an integer $s>1$ satisfies $\| \partial^{\bm{\alpha}} f  \|_{L^{\infty}([0,1]^d)} < 1$ for any $\bm{\alpha} \in \mathbb{N}^d$ with $|\xal| \leq s$. For any $N,L \in \mathbb{N}^+$, there exists a $\sigma_1$-$\tn{NN}$ $\phi$ with width $16s^{d+1}d(N+2)\log_2{(8N)}$ and depth $27s^2 (L+2)\log_2(4L)$ such that {$\|{\phi}\|_{\mathcal{W}^{1,\infty}([0,1]^d)} \leq 432s^d$ and }
\[
\| f - \phi \|_{\mathcal{W}^{1,\infty}([0,1]^d \backslash {\Omega{([0,1]^d,K,\delta)}})} \leq   84(s+1)^d 8^{s}N^{-2(s-1)/d}L^{-2(s-1)/d},
\]
where $K = \lfloor N^{1/d} \rfloor^2 \lfloor L^{2/d} \rfloor$ and $0 < \delta \leq \frac{1}{3K}$.
\end{theorem}

{We will now show our main result, Theorem \ref{coro:function_sigma_1}, provided Theorem \ref{thm:function_sigma_1} holds true. The proof of Theorem \ref{thm:function_sigma_1} will be given in Section \ref{subsect:proof_thm_sigma_1}.}
{
\begin{proof}[Proof of Theorem \ref{coro:function_sigma_1}]
When $f$ is a constant function, the statement is trivial.
By Theorem \ref{thm:function_sigma_1}, there exists a $\sigma_1$-$\tn{NN}$ $\phi$ with width $16s^{d+1}d(N+2)\log_2{(8N)}$ and depth $27s^2 (L+2)\log_2(4L)$ such that $\|{\phi}\|_{\mathcal{W}^{1,\infty}([0,1]^d)} \leq 432s^d$ and
\[
\| f - \phi \|_{\mathcal{W}^{1,\infty}([0,1]^d \backslash \Omega{([0,1]^d,K,\delta)})} \leq   84(s+1)^d 8^{s}N^{-2(s-1)/d}L^{-2(s-1)/d}.
\]
Now, we set $K=\lfloor N^{1/d} \rfloor^2 \lfloor L^{2/d}\rfloor$ and choose a small $\delta$ such that
\[
Kd\delta 432^p \leq (N^{-2(s-1)/d}L^{-2(s-1)/d})^p.
\]
Then, we have

\begin{align*}
\| f - \phi \|^p_{\mathcal{W}^{1,p}([0,1]^d)} &= \| f - \phi \|^p_{\mathcal{W}^{1,p}(\Omega{([0,1]^d,K,\delta)})} +   \| f - \phi \|^p_{\mathcal{W}^{1,p}([0,1]^d \backslash \Omega{([0,1]^d,K,\delta)})}\\
&\leq  Kd\delta(432s^d)^p +  (84(s+1)^d 8^{s}N^{-2(s-1)/d}L^{-2(s-1)/d})^p\\
&\leq  \big(s^d N^{-2(s-1)/d}L^{-2(s-1)/d}\big)^p +  \big(84(s+1)^d 8^{s}N^{-2(s-1)/d}L^{-2(s-1)/d}\big)^p\\
&\leq  (85(s+1)^d 8^{s}N^{-2(s-1)/d}L^{-2(s-1)/d})^p.
\end{align*}
Hence, we have
\[
\| f - \phi \|_{\mathcal{W}^{1,p}([0,1]^d)} \leq 85(s+1)^d 8^{s}N^{-2(s-1)/d}L^{-2(s-1)/d}.
\]
\qed
\end{proof}
}

Before showing {Theorem \ref{thm:function_sigma_1}} directly, we present the following lemmas which will be helpful to understand the different steps illustrated in the proof. Similar to showing \cite[Theorem 2.2]{lu2020deep} where only the $L^{\infty}$ norm is considered, in order to prove Theorem \ref{thm:function_sigma_1} we will first construct ReLU networks to approximate multivariate polynomials and provide an error bound measured in the $\mathcal{W}^{1,\infty}$ norm that is explicitly characterized in terms of the layer and depth of the underlying ReLU networks, via the following steps:

\begin{enumerate}
\item We approximate $f(x)=x^2$ by the compositions of the ``sawtooth functions'', which was first proposed in \cite{Yarotsky2017}.

\item We approximate $f(x,y)=xy$ using the ReLU network constructed in the previous step based on the identity $xy = 2\Big( \big( \frac{|x+y|}{2} \big)^2 - \big( \frac{|x|}{2} \big)^2 -\big( \frac{|y|}{2} \big)^2  \Big)$.

\item We approximate $f(x_1,x_2,\cdots,x_k)=x_1 x_2 \cdots x_k$ for $k \geq 2$ repeatedly using the ReLU networks constructed in the previous step.

\item We approximate a general polynomial $\bm{x}^{\bm{\alpha}}=x_1^{\alpha_1} x_2^{\alpha_2} \cdots x_d^{\alpha_d}$ for $\bm{\alpha} \in \mathbb{N}^d$ with $|\xal| \leq k \in\mathbb{N}^+ $. Since any one term polynomial of degree less than or equal to $k$ can be written as $Cz_1 z_2 \cdots z_k$, where $C$ is a constant, we can obtain an error bound based on the previous approximation result.
\end{enumerate}

We begin with giving an approximation result for the simple function of $x^2$.

\begin{lemma}\label{lem:square_sigma1_wnorm}
For any $N,L\in \mathbb{N}^{+}$, there exists a $\sigma_1$-$\tn{NN}$ $\phi$ with width $3N$ and depth $2L$ such that {$\| \phi(x) \|_{\mathcal{W}^{1,\infty}([0,1])} \leq 2$ and} 
\begin{equation*}
||\phi(x)-x^2||_{\mathcal{W}^{1,\infty}((0,1))} \leq N^{-L}.
\end{equation*}
\end{lemma}

\begin{proof}
As in the proof of Lemma \ref{lem:NNex1} (iii), we begin by defining the piecewise linear function $f_s:[0,1] \to [0,1]$ for $s \in \mathbb{N}^+$, $s=1,2,\dots$, satisfying the following properties  

\begin{enumerate}
\item $f_s(x)=x^2$ at a set of break points $\big\{\frac{j}{2^s}:j=0,1,2,\dots,2^s \big\}$.
\item $f_s(x)$ is linear between any two adjacent break points.
 \end{enumerate}
{ For any integer $N$, let $k \in \mathbb{N}^+$ be the unique number such that $(k-1)2^{k-1}+1 \leq N \leq k2^k$. For such an $N$, $k$, and any $L'$, let $s=L'k$.} It is shown in the proof of {\cite[Lemma 5.1]{lu2020deep}} that there exists a $\sigma_1$-$\tn{NN}$ $\phi(x){=f_s(x)=f_{L'k}(x)}$ with width $3N$ and depth $L'$ such that 
\begin{equation}\label{eqn:num0}
\|\phi(x)-x^2\|_{L^{\infty}([0,1])} { =} \|f_{L'k}-x^2\|_{L^{\infty}([0,1])} \leq 2^{-2(L'k+1)} \leq 2^{-2L'k} \leq N^{-L'}.
\end{equation}

We now show that the approximation error of the first order (weak) derivative can be measured in a similar fashion.

Note that for all $x\in \big(\frac{j}{2^s}, \frac{j+1}{2^s}\big)$,
\begin{equation}\label{eqn:num1}
\phi(x)=\Big(\frac{(j+1)^2}{2^s}-\frac{j^2}{2^s}\Big)\Big(x-\frac{j}{2^s}\Big)+\Big(\frac{j}{2^s}\Big)^2.
\end{equation}

From (\ref{eqn:num1}), we have for each $j=0,1,\dots,2^s-1$,
\begin{eqnarray}\nonumber
|\phi(x)-x^2|_{\mathcal{W}^{1,\infty} \big( \big(\frac{j}{2^s}, \frac{j+1}{2^s} \big)\big)}&=&\Big\|\frac{(j+1)^2}{2^s}-\frac{j^2}{2^s}-2x\Big\|_{L^{\infty} \big(\big(\frac{j}{2^s}, \frac{j+1}{2^s} \big)\big)}\\\nonumber
&=& \Big\| \frac{2j+1}{2^s} -2x \Big\|_{L^{\infty}\big(\big(\frac{j}{2^s}, \frac{j+1}{2^s}\big)\big)}\\\nonumber
&=& \max \Big\{ \big|\frac{2j+1}{2^s} - 2\big(\frac{j}{2^s}\big)\big|,\big|\frac{2j+1}{2^s} -2\big(\frac{j+1}{2^s}\big) \big| \Big\}\\\nonumber
&=&2^{-s}\\\nonumber
&=&2^{-L'k}.
\end{eqnarray}

Clearly, from (\ref{eqn:num0}) \& (\ref{eqn:num1}) we have
\begin{equation*}
\|\phi(x)-x^2\|_{\mathcal{W}^{1,\infty}((0,1))} { =}  \|f_{L'k}-x^2\|_{\mathcal{W}^{1,\infty}((0,1))} \leq 2^{-L'k} \leq N^{-L'/2}.
\end{equation*}
{
Finally, we have
\begin{eqnarray}\nonumber
\| \phi(x) \|_{\mathcal{W}^{1,\infty}([0,1])} \leq  \Big\| \frac{ (j+1)^2 - j^2 }{2^s} \Big\|_{L^{\infty}\big(\big(\frac{j}{2^s}, \frac{j+1}{2^s}\big)\big)} =\frac{2j+1}{2^s} &\leq&\frac{2(2^s-1)+1}{2^s} \\\nonumber
&=& 2 -\frac{1}{2^{L'k}}\\\nonumber
&\leq& 2.
\end{eqnarray}
Setting $L'=2L$, we have the desired $\phi(x)$ and hence the proof is finished.
}
\qed

\end{proof}

\begin{lemma}\label{lem:xy_sigma1_wnorm}
For any $N,L\in \mathbb{N}^{+}$, there exists a $\sigma_1$-$\tn{NN}$ $\phi$ with width $9N$ and depth $2L$ such that {$\|\phi\|_{\mathcal{W}^{1,\infty}((0,1)^2)} \leq 12$ and }
\begin{equation*}
\|\phi(x,y)-xy\|_{\mathcal{W}^{1,\infty}((0,1)^2)} \leq  6N^{-L}.
\end{equation*}
\end{lemma}

\begin{proof}
By Lemma \ref{lem:square_sigma1_wnorm}, there exists a $\sigma_1$-$\tn{NN}$ $\psi$ with width $3N$ and depth $2L$ such that {$\|\psi\|_{\mathcal{W}^{1,\infty}((0,1))} \leq 2$ and }
\[
\|z^2-\psi(z)\|_{\mathcal{W}^{1,\infty}((0,1))} \leq N^{-L}.
\]

Combining the above inequality and the fact that for any $x,y \in \mathbb{R}$
\[
xy = 2\Big( \big( \frac{|x+y|}{2} \big)^2 - \big( \frac{|x|}{2} \big)^2 -\big( \frac{|y|}{2} \big)^2  \Big),
\]
we construct the following network $\phi$
\[
\phi(x,y )= 2\Big(\psi\big( \frac{|x+y|}{2}\big) - \psi\big(\frac{|x|}{2}\big) - \psi \big(\frac{|y|}{2} \big)\Big).
\]

We have
\begin{eqnarray}\nonumber
&&\| \phi(x,y)- xy \|_{\mathcal{W}^{1,\infty}((0,1)^2)} \\\nonumber
&\leq&  2 \big\|\psi\big( \frac{|x+y|}{2}\big) - ( \frac{|x+y|}{2} )^2 \big\|_{\mathcal{W}^{1,\infty}((0,1)^2)}  + 2 \big\|\psi\big( \frac{|x|}{2}\big) - ( \frac{|x|}{2} )^2 \big\|_{\mathcal{W}^{1,\infty}((0,1)^2)} \\\nonumber
& &+ 2 \big\|\psi\big( \frac{|y|}{2}\big) - ( \frac{|y|}{2} )^2 \big\|_{\mathcal{W}^{1,\infty}((0,1)^2)} \\\nonumber
&\leq&  2N^{-L} + 2N^{-L} + 2N^{-L} \\\nonumber
&=& 6N^{-L}.
\end{eqnarray}
{
Finally, we have
\begin{eqnarray}\nonumber
\| \phi(x,y)\|_{\mathcal{W}^{1,\infty}((0,1)^2)} &=& \Big\| 2\Big(\psi\big( \frac{|x+y|}{2}\big) - \psi\big(\frac{|x|}{2}\big) - \psi \big(\frac{|y|}{2} \big)\Big) \Big\|_{\mathcal{W}^{1,\infty}((0,1)^2)} \\\nonumber
&\leq& 12.
\end{eqnarray}
}
\qed
\end{proof}

By rescaling, we have the following modification of Lemma \ref{lem:xy_sigma1_wnorm}.

\begin{lemma}\label{lem:xy_sigma1_wnorm_mod}
For any $N,L\in \mathbb{N}^{+}$ and $a,b \in \mathbb{R}$ with $a<b$, there exists a $\sigma_1$-$\tn{NN}$ $\phi$ with width $9N+1$ and depth $2L$ such that {$\|\phi\|_{\mathcal{W}^{1,\infty}((a,b)^2)} \leq 12(b-a)^2$ and }
\begin{equation*}
\|\phi(x,y)-xy\|_{\mathcal{W}^{1,\infty}((a,b)^2)} \leq  6(b-a)^2N^{-L}.
\end{equation*}
\end{lemma}

\begin{proof}
By Lemma \ref{lem:xy_sigma1_wnorm}, there exists a $\sigma_1$-$\tn{NN}$ $\psi$ with width $9N$ and depth $2L$ such that {$\|\psi\|_{\mathcal{W}^{1,\infty}((0,1)^2)} \leq 12$ and }
\begin{equation*}
\|\psi(\widetilde{x},\widetilde{y})-\widetilde{x}\widetilde{y}\|_{\mathcal{W}^{1,\infty}((0,1)^2)} \leq  6N^{-L}.
\end{equation*}

By setting $x=a+(b-a)\widetilde{x}$ and $y=a+(b-a)\widetilde{y}$ for any $\widetilde{x}, \widetilde{y} \in (0,1)$, we define the following network $\phi$
\[
\phi(x,y)=(b-a)^2\psi \big(\frac{x-a}{b-a}, \frac{y-a}{b-a} \big)+a(x-a)+ a(y-a)+a^2.
\]
{Note that $a(x-a)+ a(y-a)$ is positive. Hence, the width of $\phi$ can be as small as $9N+1$.} Thus, by $xy=(b-a)^2(\frac{x-a}{b-a} \cdot \frac{y-a}{b-a} ) + a(x-a)+ a(y-a)+a^2$, we have
\begin{eqnarray}\nonumber
\|\phi(x,y)-xy\|_{\mathcal{W}^{1,\infty}((a,b)^2)} &=& (b-a)^2 \big\| \psi \big(\frac{x-a}{b-a}, \frac{y-a}{b-a} \big) - \big(\frac{x-a}{b-a}\cdot \frac{y-a}{b-a}\big)  \big\|_{\mathcal{W}^{1,\infty}((a,b)^2)}  \\\nonumber
&\leq& 6(b-a)^2N^{-L}.
\end{eqnarray}

{
Finally, we have
\begin{eqnarray}\nonumber
\| \phi(x,y)\|_{\mathcal{W}^{1,\infty}((a,b)^2)} &=& \Big\| (b-a)^2\psi \big(\frac{x-a}{b-a}, \frac{y-a}{b-a} \big)+a(x-a)+ a(y-a)+a^2 \Big\|_{\mathcal{W}^{1,\infty}((a,b)^2)} \\\nonumber
&\leq& 12(b-a)^2.
\end{eqnarray}
}
\qed
\end{proof}

\begin{lemma}\label{lem:product_sigma1_wnorm}
For any $N,L, k \in \mathbb{N}^{+}$ with $ k \geq 2 $, there exists a $\sigma_1$-$\tn{NN}$ $\phi$ with width $9(N+1)+k-1$ and depth $14k(k-1)L$ such that {$\|\phi\|_{\mathcal{W}^{1,\infty}((0,1)^k)} \leq 18$ and }
\begin{equation*}
\|\phi(\bm{x})-x_1x_2\cdots x_k\|_{\mathcal{W}^{1,\infty}((0,1)^k)} \leq  10(k-1)(N+1)^{-7kL}.
\end{equation*}
\end{lemma}

\begin{proof}
 By Lemma \ref{lem:xy_sigma1_wnorm_mod}, there exists a $\sigma_1$-$\tn{NN}$ $\phi_1$ with width $9(N+1)+1$ and depth $14kL$ such that {$\|\phi_1\|_{\mathcal{W}^{1,\infty}((-0.1,1.1)^2))} \leq 18$ and }
\begin{eqnarray*}
\|\phi_1(x,y)-xy\|_{\mathcal{W}^{1,\infty}((-0.1,1.1)^2)} &\leq&  6(1.2)^2(N+1)^{-7kL}\\
&\leq& 9(N+1)^{-7kL}.
\end{eqnarray*}

Now, our goal is to construct via induction that for $i=1,2,\dots,k-1$ there exists a $\sigma_1$-$\tn{NN}$ $\phi_i$ with width $9(N+1)+i$ and depth $14kiL$ such that
\begin{equation*}
\|\phi_i( x_1, \cdots,x_{i+1} )-x_1x_2\cdots x_{i+1}\|_{\mathcal{W}^{1,\infty}((0,1)^{i+1})} \leq  10i(N+1)^{-7kL},
\end{equation*}
for any $[x_1,x_2,\cdots,x_{i+1}]^T \in (0,1)^{i+1}.$

When $d=1$, $\phi_1$ satisfies the condition.

Assuming now for any $i \in \{1,2,\dots, d-1\}$ there exists a $\sigma_1$-$\tn{NN}$ $\phi_i$ such that the both conditions hold, we define $\phi_{i+1}$ as follows:
\[
\phi_{i+1}(x_1, \cdots,x_{i+2}) = \phi_1(\phi_i(x_1,\cdots, x_{i+1}), \sigma(x_{i+2}))
\]
for any $x_1,\cdots,x_{k+2} \in \mathbb{R}$. { We can shift $x_{i+2}$ to obtain a nonnegative number $x_{i+2}+0.1$, which can be copied via a ReLU network of width $1$ to the input of $\phi_1$. Before inputting $x_{i+2}+0.1$ to $\phi_1$, we can shift it back to $x_{i+2}$.} Therefore, $\phi_{i+1}$ can be implemented via a $\sigma_1$-$\tn{NN}$ with width $9(N+1)+i+1$ and depth $14kiL + 14kL = 14k(i + 1)L$.

By the induction assumption, we have
\[
\|\phi_i( x_1,x_2, \cdots,x_{i+1} )-x_1x_2\cdots x_{i+1}\|_{\mathcal{W}^{1,\infty}((0,1)^{i+1})} \leq  10i(N+1)^{-7kL}.
\]
Note that $10i(N+1)^{-7k{ L}} \leq 10k2^{-7k} < 10k\frac{1}{100k} = 0.1$ for any $N,L,k \geq d \in \mathbb{N}^+$ and $i\in \{1,2,\cdots,d-1 \}$. Thus, we have
\[
\phi_i(x_1,x_2,\cdots,x_{i+1}) \in (-0.1, 1.1)
\]
and 
\[
 \frac{\partial \phi_i }{\partial x_1}  \in (-0.1, 1.1)
\]
for any $x_1,x_2,\cdots,x_{i+1} \in (0,1)$.

Hence, for any $x_1,x_2,\cdots,x_{i+2} \in (0,1)$, we have
\begin{eqnarray}\nonumber
&&\| \phi_{i+1}(x_1, \cdots,x_{i+2}) - x_1\cdots x_{i+2} \|_{\mathcal{W}^{1,\infty}((0,1)^{i+2})} \\\nonumber
&=& \| \phi_1(\phi_i(x_1,\cdots, x_{i+1}), \sigma(x_{i+2})) - x_1\cdots x_{i+2}\|_{\mathcal{W}^{1,\infty}((0,1)^{i+2})}\\\nonumber
&\leq&  \|  \phi_1(\phi_i(x_1,\cdots, x_{i+1}), x_{i+2})  - \phi_i(x_1,\cdots, x_{i+1})x_{i+2} \|_{\mathcal{W}^{1,\infty}((0,1)^{i+2})} \\\nonumber
&&+ \| \phi_i(x_1,\cdots, x_{i+1})x_{i+2}-x_1\cdots x_{i+2}  \|_{\mathcal{W}^{1,\infty}((0,1)^{i+2})}\\\nonumber
&\leq&  { 10(N+1)^{-7kL} } + 10i(N+1)^{-7kL}\\\label{ineqn:phi_i}
&=&  10(i+1)(N+1)^{-7kL},
\end{eqnarray}

where the second inequality is obtained since we have

\begin{eqnarray*}
\Big| \frac{\partial (\phi_1(\phi_i(x_1,\cdots, x_{i+1}), x_{i+2})  - \phi_i(x_1,\cdots, x_{i+1})x_{i+2})}{\partial x_1} \Big|&=& \Big| \frac{\partial \phi_1(\phi_i(x_1,\cdots, x_{i+1}), x_{i+2})}{\partial \phi_i} \cdot \frac{\partial \phi_i}{\partial x_1}  - x_{i+2} \cdot \frac{\partial \phi_i }{\partial x_1}  \Big| \\
 &=&  \underbrace{\Big| \frac{\partial \phi_1 (\phi_i(x_1,\cdots, x_{i+1}), x_{i+2})}{\partial \phi_i}   - x_{i+2}\Big|}_{\leq 9(N+1)^{-7kL}}   \underbrace{ \Big| \frac{\partial \phi_i }{\partial x_1} \Big|}_{<=1.1} \\
&\leq & 10(N+1)^{-7kL}
\end{eqnarray*}

and

\begin{eqnarray*}
\Big| \frac{\partial (\phi_1(\phi_i(x_1,\cdots, x_{i+1}), x_{i+2})  - \phi_i(x_1,\cdots, x_{i+1})x_{i+2})}{\partial x_{i+2}} \Big|&=& \Big| \frac{\partial \phi_1(\phi_i(x_1,\cdots, x_{i+1}), x_{i+2})}{\partial x_{i+2}}  - \phi_i  \Big| \\
&\leq & 9(N+1)^{-7kL}\\
& \leq& 10(N+1)^{-7kL} .
\end{eqnarray*}

Thus,
\[
|\phi_i(x_1,\cdots, x_{i+1})x_{i+2}-x_1\cdots x_{i+2}|_{\mathcal{W}^{1,\infty}((0,1)^{i+2})}\leq  10(N+1)^{-7kL}.
\] Also, using similar arguments or see the proof of \cite[Lemma 5.1]{lu2020deep}, it can be shown that
\[
|\phi_i(x_1,\cdots, x_{i+1})x_{i+2}-x_1\cdots x_{i+2} |_{L^{\infty}((0,1)^{i+2})}\leq  10(N+1)^{-14kL}.
\] Hence, we have shown \[\|  \phi_1(\phi_i(x_1,\cdots, x_{i+1}), x_{i+2})  - \phi_i(x_1,\cdots, x_{i+1})x_{i+2} \|_{\mathcal{W}^{1,\infty}((0,1)^{i+2})} \leq { 10(N+1)^{-7kL} },\]
which was used in showing (\ref{ineqn:phi_i}).

{
Also, for any $x_1,x_2,\cdots,x_{i+2} \in (0,1)$, we have
\begin{eqnarray}\nonumber
\| \phi_{i+1}(x_1, \cdots,x_{i+2}) \|_{\mathcal{W}^{1,\infty}((0,1)^{i+2})} &=& \| \phi_1(\phi_i(x_1,\cdots, x_{i+1}), \sigma(x_{i+2})) \|_{\mathcal{W}^{1,\infty}((0,1)^{i+2})}\\\nonumber
&\leq& 18.
\end{eqnarray}
}
At last, setting $\phi=\phi_{d-1}$ in (\ref{ineqn:phi_i}), we have finished the proof by induction.
\qed
\end{proof}

\begin{proposition}\label{prop:poly_sigma_1}
Suppose $\bm{x}^{\bm{\alpha}}=x_1^{\alpha_1} x_2^{\alpha_2} \cdots x_d^{\alpha_d}$ for $\bm{\alpha} \in \mathbb{N}^d$ with $|\xal| \leq k \in\mathbb{N}^+ $. For any $N,L \in \mathbb{N}^{+}$, there exists a $\sigma_1$-$\tn{NN}$ $\phi$ with width $9(N+1)+k-1$ and depth $14k^2L$ such that {$\|\phi\|_{\mathcal{W}^{1,\infty}([0,1]^d)} \leq 18$ and }
\begin{equation*}
\|\phi(\bm{x})-\bm{x}^{\bm{\alpha}}\|_{\mathcal{W}^{1,\infty}([0,1]^d)} \leq 10k(N+1)^{-7kL}.
\end{equation*}
\end{proposition}

\begin{proof}
This proof is similar to that of \cite[Proposition 4.1]{lu2020deep}, but for readability we provide a detailed proof as follows.

The case when $k=1$ is trivial. When $k\geq 2$, we set $\widetilde{k}= |\xal| \leq k$, $\bm{\alpha}=[\alpha_1,\alpha_2,\dots, \alpha_d]^T \in \mathbb{R}^d$, and let $ [z_1,z_2,\cdots, z_{\widetilde{k}}]^T \in \mathbb{R}^{\widetilde{k}}$ be the vector such that
\[
z_l = x_j,\quad \text{if}~ \sum_{i=1}^{j-1}\alpha_i < l \leq \sum_{i=1}^j \leq \alpha_j, \quad \text{for} ~j=1,2,\cdots,d.
\]
In other words, we have
\[
[z_1,z_2,\cdots, z_{\widetilde{k}}]^T=[\underbrace{x_1,\cdots,x_1}_{\alpha_1~\text{times}}, \underbrace{x_2,\cdots,x_2}_{\alpha_2~\text{times}}, \cdots,\underbrace{x_d,\cdots,x_d}_{\alpha_d~\text{times}}]^T \in \mathbb{R}^{\widetilde{k}}.
\]

We now construct the target deep ReLU neural network. First, there exists a linear map $\mathcal{L}:\mathbb{R}^d \to \mathbb{R}^k$ $\mathcal{L}(\xx) = [z_1,z_2,\cdots,z_{\widetilde{k}}, 1, \cdots, 1]$, which copies $\xx$ to form a new vector $[z_1,z_2,\cdots,z_{\widetilde{k}}, 1, \cdots, 1] \in \mathbb{R}^k$. Second, there exists by Lemma \ref{lem:product_sigma1_wnorm} a function $\psi : \mathbb{R}^k \rightarrow \mathbb{R}$ implemented by a ReLU network with width $9(N+1)+k-1$ and depth $14k(k-1)L$ such that {$\|\psi\|_{\mathcal{W}^{1,\infty}([0,1]^d)} \leq 18$ and} $\psi$ maps $[z_1,z_2,\cdots,z_{\widetilde{k}}, 1, \cdots, 1]$ to $z_1 z_2 \cdots z_{\widetilde{k}}$ within an error of $10(k-1)(N+1)^{-7kL}$.  Thus, we can construct the network $\phi = \psi \circ \mathcal{L}$. Then, $\phi$ can be implemented by a ReLU with width $9(N+1)+k-1$ and depth $14k(k-1)L \leq 14k^2L$, and 
\begin{eqnarray*}
\|\phi(\bm{x})-\bm{x}^{\bm{\alpha}}\|_{\mathcal{W}^{1,\infty}([0,1]^d)} &=& \| \psi \circ \mathcal{L} ( \bm{x} )- x_1^{\alpha_1} x_2^{\alpha_2} \cdots x_d^{\alpha_d} \|_{\mathcal{W}^{1,\infty}([0,1]^d)} \\
&=&\| \psi(z_1,z_2,\cdots,z_{\widetilde{k}}, 1, \cdots, 1) - z_1 z_2 \cdots z_{\widetilde{k}} \|_{\mathcal{W}^{1,\infty}([0,1]^d)}  \\
&\leq& 10(k-1)(N+1)^{-7kL} \\
&\leq& 10k(N+1)^{-7kL}.
\end{eqnarray*}

{
Also, we have
\begin{eqnarray*}
\|\phi(\bm{x})\|_{\mathcal{W}^{1,\infty}([0,1]^d)} &=& \| \psi \circ \mathcal{L} ( \bm{x} )\|_{\mathcal{W}^{1,\infty}([0,1]^d)} \\
&=&\| \psi(z_1,z_2,\cdots,z_{\widetilde{k}}, 1, \cdots, 1)\|_{\mathcal{W}^{1,\infty}([0,1]^d)}  \\
&\leq& 18.
\end{eqnarray*}
}
This proof is then finished.
\qed

\end{proof}

We are now ready to show our main approximation result of Theorem \ref{thm:function_sigma_1}. {The main idea of the proof is based on the recurring use of the Taylor polynomial approximation to smooth functions. We first construct ReLU subnetworks to realize these polynomials and their derivatives, and then implement the target smooth functions by composing the subnetworks. The overall approximation error is estimated by the remainder terms in Taylor's theorem and the error resulted from those ReLU subnetworks for approximating the polynomials.}

\subsection{Proof of Theorem \ref{thm:function_sigma_1}}\label{subsect:proof_thm_sigma_1}

\begin{proof}

We set $K = \lfloor {N^{1/d}} \rfloor^2 \lfloor L^{2/d} \rfloor$ and let {$\Omega{([0,1]^d,K,\delta)}$} defined by (\ref{eqn:trifling_region}) partition $[0,1]^d$ into $K^d$ cubes $Q_{\bm{\beta}}$ for $\bm{\beta} \in \{0,1,\cdots, K-1\}^d$ such that {\[[0,1]^d = \Omega{([0,1]^d,K,\delta)} \bigcup \big( \cup_{\bm{\beta} \in \{0,1,\cdots,K-1\}^d}   Q_{\bm{\beta}} \big). \]} For each  $\bm{\beta}=[\beta_1,\beta_2,\cdots,\beta_d]^T \in \{0,1,\cdots,K-1\}^d$, we define 
\[
Q_{\bm{\beta}} = \Bigg\{ \bm{x}=[x_1,x_2,\cdots,x_d]^T : x_i \in \Big[ \frac{\beta_i}{K}, \frac{\beta_i+1}{K} -\delta \cdot 1_{\{ \beta_i \leq K-2 \}}(\beta_i) \Big]~\text{for}~i=1,2,\dots, d \Bigg\},
\]
where $1_{\{ \beta_i \leq K-2 \}}(\beta_i)$ is the indicator function of the set $\{ \beta_i \leq K-2 \}$. 

By Proposition \ref{prop:step_sigma_1}, there exists a $\sigma_1$-$\tn{NN}$ $\psi$ with width $4N+3$ and depth $4L+5$ such that
\[
\psi(x) = k,\quad \text{if} ~ x\in \big[\frac{k}{K},\frac{k+1}{K}- \delta \cdot 1_{\{ k \leq K-2 \}}\big]~\text{for}~k=0,1,\dots,K-1.
\]
Then, for each $\bm{\beta} \in \{0,1,\cdots,K-1\}^d, \psi(x_i)=\beta_i$ if $\bm{x} \in Q_{\bm{\beta}}~\text{for}~i=1,2,\cdots,d$.

Define
\[
\bm{\psi{(x)}}:= [\psi(x_1),\psi(x_2),\cdots, \psi(x_d)]^T/K \quad \text{for any}~\bm{x} \in [0,1]^d,
\]
then
\[
\bm{\psi{(x)}} =\bm{\beta}/K,\quad \text{if}~\bm{x} \in Q_{\bm{\beta}},\quad \text{for}~\bm{\beta} \in \{0,1,\cdots,K-1\}^d.
\]

 Now, we fix a $\bm{\beta}\in \{ 0,1,\cdots,K-1 \}^d$ throughout the proof. For any $\bm{x} \in Q_{\bm{\beta}}$, by Taylor's expansion there exists a $\xi_{\bm{x}} \in (0,1)$ such that

\begin{equation}\label{eqn:Taylor_error}
{ f(\bm{x})} = \sum_{|\xal|\leq s-1} \frac{\partial^{\bm{\alpha}}f(\bm{\psi(\bm{x})})}{\bm{\alpha}!} \bm{h}^{\bm{\alpha}} + \sum_{|\xal|= s} \frac{\partial^{\bm{\alpha}}f(\bm{\psi(\bm{x})} + \xi_{\bm{x}}\bm{h}  )}{\bm{\alpha}!} \bm{h}^{\bm{\alpha}},
\end{equation}
where $\bm{h} = \bm{x} - \bm{\psi{(\bm{x})}}$.

 By Lemma \ref{lem:xy_sigma1_wnorm_mod}, there exists a $\sigma_1$-$\tn{NN}$ $\widetilde{\phi}$ with width $9(N+1)+1$ and depth $4s(L+1)$ such that {$\|\widetilde{\phi}\|_{\mathcal{W}^{1,\infty}((-3,3)^2)} \leq 432$ and }
\begin{eqnarray}\nonumber
\|\widetilde{\phi}(x,y)-xy\|_{\mathcal{W}^{1,\infty}((-3,3)^2))} &\leq& 6(6)^2(N+1)^{-2s(L+1)}\\\nonumber
&=&  216(N+1)^{-2s(L+1)} =:\mathcal{E}_1.
\end{eqnarray}

For each $\xal \in \mathbb{N}^d$ with $|\xal| \leq s$,  by Proposition \ref{prop:poly_sigma_1} there exist $\sigma_1$-$\tn{NN}s$ $P_{\bm{\alpha}}(\bm{x})$ with width $9(N+1)+s-1 $ and depth $14s^2L$ such that
\begin{equation}\label{eqn:E2}
\|P_{\bm{\alpha}}(\bm{x})-\bm{x}^{\bm{\alpha}}\|_{\mathcal{W}^{1,\infty}([0,1]^d)} \leq 10s(N+1)^{-7sL}=:\mathcal{E}_2.
\end{equation}

For each $i=0,1,\cdots,K^d-1$, we define the bijection
\[
\bm{\eta}(i) = [\eta_1,\eta_2,\cdots,\eta_d]^T \in \{0,1,\dots,K-1 \}^d
\]
such that $\sum_{j=1}^{d}\eta_j K^{j-1}=i$. We will drop the input $i$ in $\bm{\eta}(i)$ later for simplicity. For each $\bm{\alpha} \in \mathbb{N}^d$ with $|\xal| \leq s-1$, define
\[
\xi_{\bm{\alpha},i} = \big(\partial^{\bm{\alpha}}f\big(\frac{\eta}{K}\big) +1 \big)/2.
\]
Note that $K^d = (\lfloor N^{1/d} \rfloor^2 \lfloor L^{2/d} \rfloor)^d \leq N^2L^2$ and $\xi_{\bm{\alpha},i} \in [0,1]$ for $i=0,1,\cdots,K^d-1$. By Proposition \ref{prop:point_sigma_1}, there exists a $\sigma_1$-$\tn{NN}$ $\widetilde{\phi}_{\bm{\alpha}}$ of width $16s(N+1)\log_{2}(8N)$ and depth $5(L+2)\log_{2}(4L)$ such that
\[
| \widetilde{\phi}_{\bm{\alpha}}(i) -  \xi_{\bm{\alpha},i}| \leq N^{-2s}L^{-2s}, \quad \text{for}~i=0,1,\cdots, K^d - 1 ~\text{and}~ |\xal| \leq s-1.
\]

For each $\xal \in \mathbb{N}^d$ with $|\xal| \leq s-1$, we define
\[
\phi_{\bm{\alpha}}(\bm{x}):=2\widetilde{\phi}_{\bm{\alpha}} \Bigg( \sum_{j=1}^{d} x_j { K^{j-1}}  \Bigg) -1,\quad \text{for any}~\bm{x}=[x_1,x_2,\cdots,x_d]^d \in \mathbb{R}^d.
\]It can be seen that $\phi_{\bm{\alpha}}$ is also of width $16s(N+1)\log_{2}(8N)$ and depth $5(L+2)\log_{2}(4L)$. 

Then, for each $\bm{\eta} = [\eta_1,\eta_2,\cdots,\eta_d]^T \in \{0,1,\dots,K-1 \}^d$ corresponding to $i=\sum_{j=1}^{d}\eta_j K^{j-1}$, each $\bm{\alpha} \in \mathbb{N}^d$ with $|\xal| \leq s-1$, we have
\[
\big| \phi_{\bm{\alpha}}(\frac{\bm{\eta}}{K})  - \partial^{\bm{\alpha}}f (\frac{\bm{\eta}}{K}) \big| = \Bigg| 2\widetilde{\phi}_{\bm{\alpha}} \Bigg( \sum_{j=1}^{d} x_j K^{j-1}  \Bigg) -1  - (2\xi_{\bm{\alpha},i} -1) \Bigg| = 2|\widetilde{\phi}_{\bm{\alpha}}(i) - \xi_{\bm{\alpha},i}  | \leq 2N^{-2s}L^{-2s}.
\]

From $\bm{\psi{(x)}}=\frac{\bm{\beta}}{K}$ for $\bm{x} \in Q_{\bm{\beta}}$, it follows that
\begin{eqnarray}\nonumber
\| \phi_{\bm{\alpha} }(\bm{\psi (\bm{x}) }) - \partial^{\bm{\alpha}}f(\bm{\psi (\bm{x}) }) \|_{\mathcal{W}^{1,\infty}(Q_{\bm{\beta}})}&=& \| \phi_{\bm{\alpha} }(\bm{\psi (\bm{x}) }) - \partial^{\bm{\alpha}}f(\bm{\psi (\bm{x}) }) \|_{{L}^{\infty}(Q_{\bm{\beta}})}\\\nonumber
&=&\big| \phi_{\bm{\alpha}}(\frac{\bm{\beta}}{K})  - \partial^{\bm{\alpha}}f (\frac{\bm{\beta}}{K}) \big| \\ \label{eqn:E3}
&\leq& 2N^{-2s}L^{-2s} =:\mathcal{E}_3.
\end{eqnarray}
Note that since $ \phi_{\bm{\alpha} }(\bm{\psi (\bm{x}) }) - \partial^{\bm{\alpha}}f(\bm{\psi (\bm{x}) })$ {for $\bm{x} \in Q_{\bm{\beta}}$ is constant, its weak derivative is zero, which has given us the first equality of (\ref{eqn:E3}).}

Define
\begin{equation}\label{eqn:main_nn_sigma1}
\phi(\bm{x}) = \sum_{|\xal| \leq s-1} \widetilde{\phi} \Big( \frac{\phi_{\bm{\alpha} }(\bm{\psi (\bm{x}) })}{\bm{\alpha}!}, P_{\bm{\alpha}} (\bm{h})  \Big)
\end{equation}
for any $\bm{x} \in \mathbb{R}^d$.

{Let us now first estimate the overall approximation error in semi-norm. We denote the first order derivatives of $f$ by $\partial^{\bm{\gamma}}f$ with $|\bm{\gamma}|=1$. For any $\bm{x} \in Q_{\bm{\beta}}$, by Taylor's expansion there exists a $\xi_{\bm{x}}^{\bm{\gamma}} \in (0,1)$ such that
\begin{eqnarray}\label{eqn:Taylor_error_first_order}
{ \partial^{\bm{\gamma}}f (\bm{x})} &=& \sum_{|\xal|\leq s-2} \frac{\partial^{\bm{\alpha}}\partial^{\bm{\gamma}}f(\bm{\psi(\bm{x})})}{\bm{\alpha}!} \bm{h}^{\bm{\alpha}} + \sum_{|\xal|= s-1} \frac{\partial^{\bm{\alpha}}\partial^{\bm{\gamma}}f(\bm{\psi(\bm{x})} + \xi_{\bm{x}}^{\bm{\gamma}} \bm{h}  )}{\bm{\alpha}!} \bm{h}^{\bm{\alpha}}\\\nonumber
&=& \partial^{\bm{\gamma}} \Big( \sum_{|\xal|\leq s-1} \frac{\partial^{\bm{\alpha}}f(\bm{\psi(\bm{x})})}{\bm{\alpha}!} \bm{h}^{\bm{\alpha}} \Big)+ \sum_{|\xal|= s-1} \frac{\partial^{\bm{\alpha}}\partial^{\bm{\gamma}}f(\bm{\psi(\bm{x})} + \xi_{\bm{x}}^{\bm{\gamma}} \bm{h}  )}{\bm{\alpha}!} \bm{h}^{\bm{\alpha}},
\end{eqnarray}
where $\bm{h} = \bm{x} - \bm{\psi{(\bm{x})}}$.
\\
Therefore,
\begin{eqnarray*}
&& | \phi(\bm{x}) - f(\bm{x}) |_{\mathcal{W}^{1,\infty}(Q_{\bm{\theta}})}\\
&\leq &\Big| \phi(\bm{x})- \sum_{|\xal|\leq s-1} \frac{\partial^{\bm{\alpha}}f(\bm{\psi(\bm{x})})}{\bm{\alpha}!} \bm{h}^{\bm{\alpha}}\Big|_{\mathcal{W}^{1,\infty}(Q_{\bm{\theta}})}  +  \Big| \sum_{|\xal|\leq s-1} \frac{\partial^{\bm{\alpha}}f(\bm{\psi(\bm{x})})}{\bm{\alpha}!} \bm{h}^{\bm{\alpha}}  - f(\bm{x}) \Big|_{\mathcal{W}^{1,\infty}(Q_{\bm{\theta}})}\\
&\leq &  {\sum_{|\xal| \leq s-1}  \Big| \widetilde{\phi} \Big( \frac{\phi_{\bm{\alpha} }(\bm{\psi (\bm{x}) })}{\bm{\alpha}!}, P_{\bm{\alpha}} (\bm{h})  \Big)  -  \frac{\partial^{\bm{\alpha}}f(\bm{\psi(\bm{x})})}{\bm{\alpha}!} \bm{h}^{\bm{\alpha}} \Big|_{\mathcal{W}^{1,\infty}(Q_{\bm{\beta}})}}\\
 &&+ \max_{|\bm{\gamma}|=1} \sum_{|\xal|= s-1} \Big\|   \frac{\partial^{\bm{\alpha}}\partial^{\bm{\gamma}}f(\bm{\psi(\bm{x})} + \xi_{\bm{x}}^{\bm{\gamma}}\bm{h}  )}{\bm{\alpha}!} \bm{h}^{\bm{\alpha}} \Big\|_{\mathcal{L}^{\infty}(Q_{\bm{\beta}})}\\
&\leq &  \underbrace{\sum_{|\xal| \leq s-1}  \Big\| \widetilde{\phi} \Big( \frac{\phi_{\bm{\alpha} }(\bm{\psi (\bm{x}) })}{\bm{\alpha}!}, P_{\bm{\alpha}} (\bm{h})  \Big)  -  \frac{\partial^{\bm{\alpha}}f(\bm{\psi(\bm{x})})}{\bm{\alpha}!} \bm{h}^{\bm{\alpha}} \Big\|_{\mathcal{W}^{1,\infty}(Q_{\bm{\beta}})}}_{=:E_1} \\
&&+\underbrace{ \max_{|\bm{\gamma}|=1} \sum_{|\xal|= s-1} \Big\|   \frac{\partial^{\bm{\alpha}}\partial^{\bm{\gamma}}f(\bm{\psi(\bm{x})} + \xi_{\bm{x}}^{\bm{\gamma}}\bm{h}  )}{\bm{\alpha}!} \bm{h}^{\bm{\alpha}} \Big\|_{\mathcal{L}^{\infty}(Q_{\bm{\beta}})}}_{=:E_2},
\end{eqnarray*}
where $E_2$ with $\xi_{\bm{x}}^{\bm{\gamma}} \in (0,1)$ is the remainder term resulted from Taylor's expansion of $\partial^{\bm{\gamma}}f$ in (\ref{eqn:Taylor_error_first_order}).
}

{The error term} $E_1$ is further decomposed into two parts via
\begin{eqnarray}\nonumber
E_1 &=& \sum_{|\xal| \leq s-1}  \Big\| \widetilde{\phi} \Big( \frac{\phi_{\bm{\alpha} }(\bm{\psi (\bm{x}) })}{\bm{\alpha}!}, P_{\bm{\alpha}} (\bm{h})  \Big)  -  \frac{\partial^{\bm{\alpha}}f(\bm{\psi(\bm{x})})}{\bm{\alpha}!} \bm{h}^{\bm{\alpha}} \Big\|_{\mathcal{W}^{1,\infty}(Q_{\bm{\beta}})}\\\nonumber
&\leq& \underbrace{\sum_{|\xal| \leq s-1}  \Big\| \widetilde{\phi} \Big( \frac{\phi_{\bm{\alpha} }(\bm{\psi (\bm{x}) })}{\bm{\alpha}!}, P_{\bm{\alpha}} (\bm{h})  \Big)  -  \widetilde{\phi} \Big( \frac{\partial^{\bm{\alpha}}f(\bm{\psi(\bm{x})})  }{\bm{\alpha}!}, P_{\bm{\alpha}} (\bm{h})  \Big) \Big\|_{\mathcal{W}^{1,\infty}(Q_{\bm{\beta}})}}_{=:E_{1,1}} \\\nonumber
&&+ \underbrace{\sum_{|\xal| \leq s-1}  \Big\| \widetilde{\phi} \Big( \frac{\partial^{\bm{\alpha}}f(\bm{\psi(\bm{x})})  }{\bm{\alpha}!}, P_{\bm{\alpha}} (\bm{h})  \Big)  -  \frac{\partial^{\bm{\alpha}}f(\bm{\psi(\bm{x})})}{\bm{\alpha}!} \bm{h}^{\bm{\alpha}} \Big\|_{\mathcal{W}^{1,\infty}(Q_{\bm{\beta}})}}_{=:E_{1,2}}.
\end{eqnarray}

For each $\xal \in \mathbb{R}^d$ with $|\xal| \leq s-1$ and $\bm{x} \in Q_{\bm{\beta}}$, since $\mathcal{E}_3 \in [0,2]$ and $\partial^{\bm{\alpha}}f(\bm{\psi(\bm{x})}) \in (-1,1)$ according to (\ref{eqn:E3}), we have $\phi_{\bm{\alpha} }(\bm{\psi (\bm{x}) })\in (-3,3)$ and hence $\frac{\phi_{\bm{\alpha} }(\bm{\psi (\bm{x}) })}{\bm{\alpha}!}\in (-3,3)$.

Also, we have $P_{\bm{\alpha}} (\bm{x}) \in (-2,3) \subset (-3,3)$ and $ \|P_{\bm{\alpha}} (\bm{x}) \|_{\mathcal{W}{^{1,\infty}([0,1]^d)}} \leq 3$ for any $\bm{x} \in [0,1]^d$ and $|\xal| \leq s-1$, since $\mathcal{E}_2 <2$ from (\ref{eqn:E2}). 

Hence, we can now measure $E_{1,1}$, $E_{1,2}$ and $E_2$:

\begin{eqnarray}\nonumber
E_{1,1} &=& \sum_{|\xal| \leq s-1}  \Big\| \widetilde{\phi} \Big( \frac{\phi_{\bm{\alpha} }(\bm{\psi (\bm{x}) })}{\bm{\alpha}!}, P_{\bm{\alpha}} (\bm{h})  \Big)  -  \widetilde{\phi} \Big( \frac{\partial^{\bm{\alpha}}f(\bm{\psi(\bm{x})})  }{\bm{\alpha}!}, P_{\bm{\alpha}} (\bm{h})  \Big) \Big\|_{\mathcal{W}^{1,\infty}(Q_{\bm{\beta}})}\\\nonumber
&\leq& \sum_{|\xal| \leq s-1}  \Bigg(   \underbrace{\Big\| \widetilde{\phi} \Big( \frac{\phi_{\bm{\alpha} }(\bm{\psi (\bm{x}) })}{\bm{\alpha}!}, P_{\bm{\alpha}} (\bm{h})  \Big)  -    \frac{\phi_{\bm{\alpha} }(\bm{\psi (\bm{x}) })}{\bm{\alpha}!} P_{\bm{\alpha}} (\bm{h})   \Big\|_{\mathcal{W}^{1,\infty}(Q_{\bm{\beta}})}}_{\leq \mathcal{E}_1}  \\\nonumber
&&+ \underbrace{ \Big\| \widetilde{\phi} \Big( \frac{\partial^{\bm{\alpha}}f(\bm{\psi(\bm{x})})  }{\bm{\alpha}!}, P_{\bm{\alpha}} (\bm{h})  \Big)  -   \frac{\partial^{\bm{\alpha}}f(\bm{\psi(\bm{x})})  }{\bm{\alpha}!} P_{\bm{\alpha}} (\bm{h})   \Big\|_{\mathcal{W}^{1,\infty}(Q_{\bm{\beta}})}}_{\leq \mathcal{E}_1}  \\\nonumber
&&+  \underbrace{\Big\|  \frac{\phi_{\bm{\alpha} }(\bm{\psi (\bm{x}) })}{\bm{\alpha}!} P_{\bm{\alpha}} (\bm{h})    -   \frac{\partial^{\bm{\alpha}}f(\bm{\psi(\bm{x})})  }{\bm{\alpha}!} P_{\bm{\alpha}} (\bm{h})   \Big\|_{\mathcal{W}^{1,\infty}(Q_{\bm{\beta}})}}_{\leq 3\mathcal{E}_3}  \Bigg)\\\nonumber
&\leq& \sum_{|\xal| \leq s-1} (  2\mathcal{E}_1 + 3\mathcal{E}_3)\\\nonumber
&\leq& s^d(  2\mathcal{E}_1 + 3\mathcal{E}_3).
\end{eqnarray}

Note that the last inequality is followed by the fact that $\sum_{|\xal|\leq s-1} 1 \leq \sum_{i=0}^{s-1}(i+1)^{d-1} \leq s^d$.

 Thus,
 \begin{eqnarray}\nonumber
E_{1,2} &=& \sum_{|\xal| \leq s-1}  \Big\| \widetilde{\phi} \Big( \frac{\partial^{\bm{\alpha}}f(\bm{\psi(\bm{x})})  }{\bm{\alpha}!}, P_{\bm{\alpha}} (\bm{h})  \Big)  -  \frac{\partial^{\bm{\alpha}}f(\bm{\psi(\bm{x})})}{\bm{\alpha}!} \bm{h}^{\bm{\alpha}} \Big\|_{\mathcal{W}^{1,\infty}(Q_{\bm{\beta}})}\\\nonumber
&\leq& \sum_{|\xal| \leq s-1}  \underbrace{\Big\| \widetilde{\phi} \Big( \frac{\partial^{\bm{\alpha}}f(\bm{\psi(\bm{x})})  }{\bm{\alpha}!}, P_{\bm{\alpha}} (\bm{h})  \Big)  -  \frac{\partial^{\bm{\alpha}}f(\bm{\psi(\bm{x})})}{\bm{\alpha}!} P_{\bm{\alpha}} (\bm{h}) \Big\|_{\mathcal{W}^{1,\infty}(Q_{\bm{\beta}})}}_{\leq \mathcal{E}_1} \\\nonumber
&&+ \sum_{|\xal| \leq s-1}  {\underbrace{\Big\| \frac{\partial^{\bm{\alpha}}f(\bm{\psi(\bm{x})})}{\bm{\alpha}!} P_{\bm{\alpha}} (\bm{h}) -  \frac{\partial^{\bm{\alpha}}f(\bm{\psi(\bm{x})})}{\bm{\alpha}!} \bm{h}^{\bm{\alpha}} \Big\|_{\mathcal{W}^{1,\infty}(Q_{\bm{\beta}})}}_{_{\leq \mathcal{E}_2}}}\\\nonumber
&\leq& \sum_{|\xal| \leq s-1} (\mathcal{E}_1 + \mathcal{E}_2)\\\nonumber
&\leq& s^d (\mathcal{E}_1 + \mathcal{E}_2).
\end{eqnarray}

{
We now provide a measure on $E_2$. First, by the assumption that $\| \partial^{\bm{\alpha}} f  \|_{L^{\infty}([0,1]^d)} < 1$ for any $\bm{\alpha} \in \mathbb{N}^d$ with $|\xal| \leq s$, we have
 \begin{eqnarray*}
E_2 &=& \max_{|\bm{\gamma}|=1} \sum_{|\xal|= s-1} \Big\|   \frac{\partial^{\bm{\alpha}}\partial^{\bm{\gamma}}f(\bm{\psi(\bm{x})} + \xi_{\bm{x}}^{\bm{\gamma}}\bm{h}  )}{\bm{\alpha}!} \bm{h}^{\bm{\alpha}} \Big\|_{\mathcal{L}^{\infty}(Q_{\bm{\beta}})}\\
 &\leq &\sum_{|\xal|= s-1} \Big\| \frac{1}{\bm{\alpha}!}  \bm{h}^{\bm{\alpha}} \Big\|_{{L}^{\infty}(Q_{\bm{\beta}})}\\
 &\leq & s^{d-1}K^{-(s-1)},
 \end{eqnarray*}
where the last inequality is followed by $\sum_{|\xal|=s-1} 1 \leq s^{d-1}$.
\\
Thus, we have shown that
 \begin{eqnarray*}
&&| \phi(\bm{x}) - f(\bm{x}) |_{\mathcal{W}^{1,\infty}(Q_{\bm{\beta}})} \\\nonumber
&\leq& E_{1,1} + E_{1,2} + E_2 \\\nonumber
&\leq& E_{1,1} + E_{1,2}   + s^{d-1} K^{-(s-1)}.
\end{eqnarray*}
We now estimate $\| \phi(\bm{x}) - f(\bm{x}) \|_{{L}^{\infty}(Q_{\bm{\beta}})}$. Using similar arguments, from the Taylor expansion defined in (\ref{eqn:Taylor_error}) with $\xi_{\bm{x}}$ in the remainder term, we can show that
 \begin{eqnarray*}
&&\| \phi(\bm{x}) - f(\bm{x}) \|_{{L}^{\infty}(Q_{\bm{\beta}})} \\\nonumber
&\leq& \sum_{|\xal| \leq s-1}  \Big\| \widetilde{\phi} \Big( \frac{\phi_{\bm{\alpha} }(\bm{\psi (\bm{x}) })}{\bm{\alpha}!}, P_{\bm{\alpha}} (\bm{h})  \Big)  -  \frac{\partial^{\bm{\alpha}}f(\bm{\psi(\bm{x})})}{\bm{\alpha}!} \bm{h}^{\bm{\alpha}} \Big\|_{{L}^{\infty}(Q_{\bm{\beta}})} +  \sum_{|\xal|=s}  \Big\|   \frac{\partial^{\bm{\alpha}}f(\bm{\psi(\bm{x})} + \xi_{\bm{x}}\bm{h}  )}{\bm{\alpha}!} \bm{h}^{\bm{\alpha}} \Big\|_{{L}^{\infty}(Q_{\bm{\beta}})}\\
&\leq& \sum_{|\xal| \leq s-1}  \Big\| \widetilde{\phi} \Big( \frac{\phi_{\bm{\alpha} }(\bm{\psi (\bm{x}) })}{\bm{\alpha}!}, P_{\bm{\alpha}} (\bm{h})  \Big)  -  \frac{\partial^{\bm{\alpha}}f(\bm{\psi(\bm{x})})}{\bm{\alpha}!} \bm{h}^{\bm{\alpha}} \Big\|_{\mathcal{W}^{1,\infty}(Q_{\bm{\beta}})}  + (s+1)^{d-1} K^{-s}\\
&\leq& E_{1,1} + E_{1,2}   + (s+1)^{d-1} K^{-s},
\end{eqnarray*}
where the second last inequality is due to $\sum_{|\xal|=s} 1 \leq (s+1)^{d-1}$.
}

Using $(N+1)^{-7s(L+1)} \leq (N+1)^{-2s(L+1)} \leq (N+1)^{-2s}2^{-{ 2}sL} \leq N^{-2s}L^{-2s}$ and $K=\lfloor N^{1/d} \rfloor^2 \lfloor L^{2/d} \rfloor \geq \frac{N^{2/d} L^{2/d}}{8}$, we have
{
\begin{eqnarray}\nonumber
&& \| \phi(\bm{x}) - f(\bm{x}) \|_{\mathcal{W}^{1,\infty}(Q_{\bm{\beta}})} \\\nonumber
&=& \max{\big\{ \| \phi(\bm{x}) - f(\bm{x}) \|_{{L}^{\infty}(Q_{\bm{\beta}})}   , | \phi(\bm{x}) - f(\bm{x}) |_{\mathcal{W}^{1,\infty}(Q_{\bm{\beta}})} \big\}}  \\\nonumber
&\leq& E_{1,1} + E_{1,2} + (s+1)^{d-1} K^{-(s-1)} \\\nonumber
&\leq& s^d(  2\mathcal{E}_1 + 3\mathcal{E}_3) + s^d (\mathcal{E}_1 + \mathcal{E}_2 )  + (s+1)^{d-1} K^{-(s-1)}\\\nonumber
&\leq& (s+1)^d( K^{-(s-1)} +3\mathcal{E}_1 + \mathcal{E}_2 +3\mathcal{E}_3 )\\\nonumber
&\leq& (s+1)^d\big(  K^{-(s-1)} + 648(N+1)^{-2s(L+1)} + 10s(N+1)^{-7s(L+1)} + 6N^{-2s}L^{-2s} \big)\\\nonumber
&\leq&(s+1)^d\big(  8^{s-1}N^{-2(s-1)/d}L^{-2(s-1)/d} + ( 654+10s)N^{-2s}L^{-2s} \big)\\\nonumber
&\leq&(s+1)^d  (8^{s-1}+654+10s)N^{-2(s-1)/d}L^{-2(s-1)/d}  \\\nonumber
&\leq& 84(s+1)^d 8^{s}N^{-2(s-1)/d}L^{-2(s-1)/d}.
\end{eqnarray}
}

Since $\bm{\beta} \in \{0,1,2,\cdots,K-1  \}^d$ is arbitrary and the fact that {$[0,1]^d \backslash \Omega{([0,1]^d,K,\delta)} \subseteq  \cup_{\bm{\beta} \in \{0,1,\cdots,K-1\}^d}  Q_{\bm{\beta}} $}, we have
{\[
\| \phi(\bm{x}) - f(\bm{x}) \|_{\mathcal{W}^{1,\infty}([0,1]^d \backslash \Omega{([0,1]^d,K,\delta)})} \leq  84(s+1)^d 8^{s}N^{-2(s-1)/d}L^{-2(s-1)/d}.
\]}

{
Furthermore, we have
\begin{eqnarray*}
\| \phi(\bm{x})\|_{\mathcal{W}^{1,\infty}([0,1]^d )} &=& \Bigg\|   \sum_{|\xal| \leq s-1} \widetilde{\phi} \Big( \frac{\phi_{\bm{\alpha} }(\bm{\psi (\bm{x}) })}{\bm{\alpha}!}, P_{\bm{\alpha}} (\bm{h})  \Big) \Bigg\|_{\mathcal{W}^{1,\infty}([0,1]^d )}\\
&\leq &  \sum_{|\xal| \leq s-1}  \| \widetilde{\phi}\|_{\mathcal{W}^{1,\infty}([0,1]^d)}\\
&\leq& 432s^d.
\end{eqnarray*}
}

As last, we finish the proof by estimating the width and depth of the network implementing $\phi(\bm{x})$. From (\ref{eqn:main_nn_sigma1}), we know that $\phi(\bm{x})$ consists of the following subnetworks:

\begin{enumerate}
\item $\bm{\psi} \in \tn{NN}(\textrm{width} \leq d(4N+3);\textrm{depth} \leq 4L+5)$;
\item $ \phi_{\bm{\alpha} } \in \tn{NN}(\textrm{width} \leq 16s(N+1)\log_{2}(8N);\textrm{depth} \leq 5(L+2)\log_{2}(4L))$;
\item $P_{\bm{\alpha}} \in \tn{NN}(\textrm{width} \leq 9(N+1)+s-1;\textrm{depth} \leq 14s^2L)$;
\item $\widetilde{\phi} \in \tn{NN}(\textrm{width} \leq 9(N+1)+1;\textrm{depth} \leq 4s(L+1))$.
\end{enumerate}

Thus, we can infer that the $\phi$ can be implemented by a ReLU network with width $16s^{d+1}{ d}(N+2)\log_2{(8N)}$ and depth 
\begin{equation*}
	\begin{split}
		(4L+5)+4s(L+1)+14s^2L+5(L+2)\log_2(4L) +3\le
		27s^2(L+2)\log_2(4L)
	\end{split}.
\end{equation*} Hence, we have finished the proof.
\qed
\end{proof}

\subsection{Approximation Using $\sigma_2$-Neural Networks}\label{section:signma_2}

In this subsection, we provide approximation results for smoothing functions measured in the $\mathcal{W}^{n,\infty}$ norm, where $n$ is a positive integer, using the smoother $\sigma_2$-NNs. First, we list a few basic lemmas of $\sigma_2$-NNs repeatedly applied in our main analysis.

\begin{lemma}\label{lem:NNex2}
The following basic lemmas of $\sigma_2$-$\tn{NN}$s hold:
\begin{enumerate}[label=(\roman*)]
\item $\sigma_1$-$\tn{NN}$s are $\sigma_2$-$\tn{NN}$s.
\item Any identity map in $\R^d$ can be realized exactly by a $\sigma_2$-$\tn{NN}$ with one hidden layer and $2d$ neurons.
\item $f(x)=x^2$ can be realized exactly by $\sigma_2$-$\tn{NN}$ with one hidden layer and two neurons.
\item $f(x,y)=xy=\frac{(x+y)^2-(x-y)^2}{4}$ can be realized exactly by $\sigma_2$-$\tn{NN}$ with one hidden layer and four neurons.
\item Assume $\xx^\xal=x_1^{\alpha_1}x_2^{\alpha_2}\cdots x_d^{\alpha_d}$ for $\xal\in \N^d$. For any $N,L\in \N^+$ such that $NL+2^{\lfloor \log_2 N \rfloor}\geq |\xal|$, there exists a $\sigma_2$-$\tn{NN}$ $\phi$ with width $4N+2d$ and depth $L+\lceil \log_2 N\rceil$ such that
    \begin{equation*}
    \phi(\xx)=\xx^\xal \quad \tn{for any $\xx\in \R^d$.}
    \end{equation*} 
    \item Assume $P(\xx)= \sum_{j=1}^J c_j \xx^{\xal_j}$ for $\xal_j\in \N^d$. For any $N,L,a,b\in \N^+$ such that $ab\geq J$ and $(L-2b-b\log_2 N)N\geq  b \max_j |\xal_j|$, there exists a $\sigma_2$-$\tn{NN}$ $\phi$ with width $4Na+2d+2$ and depth $L$ such that
    \begin{equation*}
 \phi(\xx)=P(\xx)\quad \tn{for any $\xx\in \R^d$.}
    \end{equation*}
\end{enumerate}
\end{lemma}

\begin{proof}
Showing (i) to (iv) is trivial. We will only prove (v) and (vi) in the following.

Part (v): In the case of $|\xal|=k\leq 1$, the proof is simple and left for the reader. When $|\xal|=k\ge 2$, the main idea of the proof of (v) can be summarized in Figure \ref{fig:NN1}. We apply $\sigma_1$-$\tn{NN}$s to implement a $d$-dimensional identity map as in Lemma \ref{lem:NNex1} (iii). These identity maps maintain necessary entries of $\xx$ to be multiplied together. We apply $\sigma_2$-$\tn{NN}$s to implement the multiplication function in Lemma \ref{lem:NNex2} (iii) and carry out the multiplication $N$ times per layer. After $L$ layers, there are $k-NL\leq N$ multiplication to be implemented. Finally, these at most $N$ multiplications can be carried out with a small $\sigma_2$-$\tn{NN}$s in a dyadic tree structure.

Part (vi): The main idea of the proof is to apply Part (v) $J$ times to construct $J$ $\sigma_2$-$\tn{NN}$s, $\{\phi_j(\xx)\}_{j=1}^J$, to represent $\xx^{\xal_j}$ and arrange these $\sigma_2$-$\tn{NN}$s as sub-NN blocks to form a larger $\sigma_2$-$\tn{NN}$ $\tilde{\phi}(\xx)$ with $ab$ blocks as shown in Figure \ref{fig:NN2}, where each red rectangle represents one $\sigma_2$-$\tn{NN}$ $\phi_j(\xx)$ and each blue rectangle represents one $\sigma_1$-$\tn{NN}$ of width $2$ as an identity map of $\R$. There are $ab$ red blocks with $a$ rows and $b$ columns. When $ab\geq J$, these sub-NN blocks can carry out all monomials $\xx^{\xal_j}$. In each column, the results of the multiplications of $\xx^{\xal_j}$ are added up to the input of the narrow $\sigma_1$-$\tn{NN}$, which can carry the sum over to the next column. After the calculation of $b$ columns, $J$ additions of the monomials $\xx^{\xal_j}$ have been implemented, resulting in the output $P(\xx)$. 

By Part (v), for any $N\in \N^+$, there exists a $\sigma_2$-$\tn{NN}$ $\phi_j(\xx)$ of width $2d+4N$ and depth $L_j = \lceil \frac{|\xal_j|}{N}\rceil +\lceil \log_2 N \rceil$ to implement $\xx^{\xal_j}$. Note that $b\max_j L_j\leq b\left(  \frac{\max_j |\xal_j|}{N} +2 + \log_2 N \right)$. Hence, there exists a $\sigma_2$-$\tn{NN}$ $\tilde{\phi}(\xx)$ of width $2da+4Na+2$ and depth $b\left(  \frac{\max_j |\xal_j|}{N} + 2 + \log_2 N \right)$ to implement $P(\xx)$ as in Figure \ref{fig:NN2}. Note that the total width of each column of blocks is $2ad+4Na+2$ but in fact this width can be reduced to $2d+4Na+2$, since the red blocks in each column can share the same identity map of $\R^d$ (the blue part of Figure \ref{fig:NN1}).

Note that $b\left(  \frac{\max_j |\xal_j|}{N} + 2 + \log_2 N \right)\leq L$ is equivalent to $(L-2b-b\log_2 N)N\geq b \max_j |\xal_j|$. Hence, for any $N,L,a,b\in \N^+$ such that $ab\geq J$ and $(L-2b-b\log_2 N)N\geq b \max_j |\xal_j|$, there exists a $\sigma_2$-$\tn{NN}$ $\phi(\xx)$ with width $4Na+2d+2$ and depth $L$ such that $\tilde{\phi}(\xx)$ is a sub-NN of $\phi(\xx)$ in the sense of $\phi(\xx)=\tn{Id}\circ\tilde{\phi}(\xx)$ with $\tn{Id}$ as an identify map of $\R$, which means that $\phi(\xx)=\tilde{\phi}(\xx)=P(\xx)$. The proof of Part (vi) is completed.
\qed
\end{proof}

	\begin{figure}[!ht]
		\centering
		\includegraphics[width=0.8\linewidth]{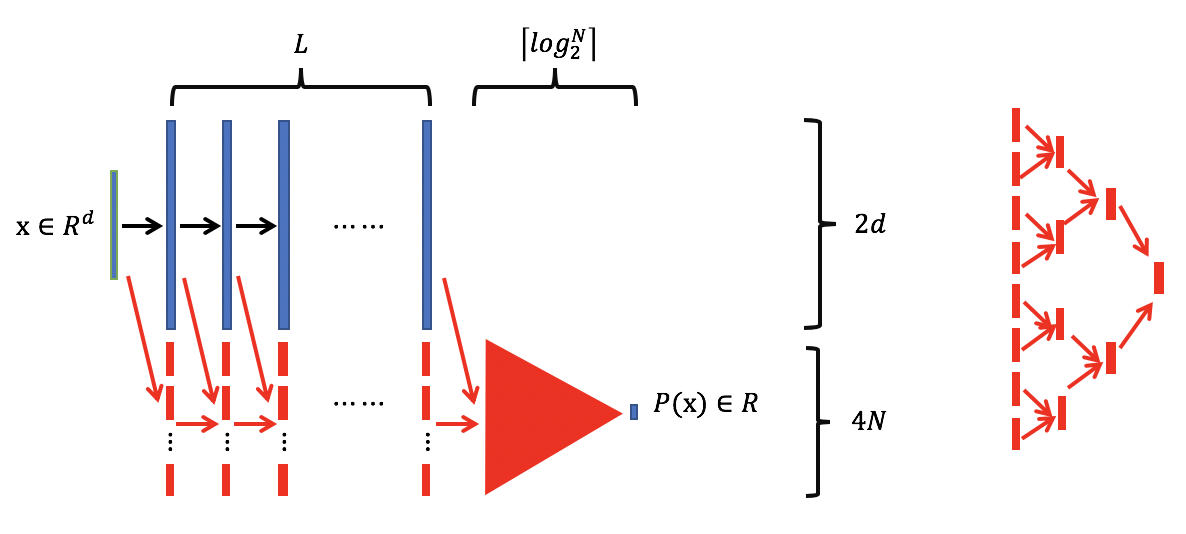}
		\caption{Left: An illustration of the proof of Lemma \ref{lem:NNex2} (v). Green vectors represent the input and output of the $\sigma_2$-$\tn{NN}$ carrying out $P(\xx)$. Blue vectors represent the $\sigma_1$-$\tn{NN}$ that implements a $d$-dimensional identity map in Lemma \ref{lem:NNex1} (iii), which was repeatedly applied for $L$ times. Black arrows represent the data flow for carrying out the identity maps. Red vectors represent the $\sigma_2$-$\tn{NN}$s implementing the multiplication function in Lemma \ref{lem:NNex2} (iii) and there $NL$ such red vectors. Red arrows represent the data flow for carrying out the multiplications. Finally, a red triangle represent a $\sigma_2$-$\tn{NN}$ of width at most $4N$ and depth at most $\lceil \log_2^N \rceil$ carrying out the rest of the multiplications. Right: An example of the red triangle is given on the right when it consists of $15$ red vectors carrying out $15$ multiplications.}
		\label{fig:NN1}
	\end{figure}

	\begin{figure}[!ht]
		\centering
		\includegraphics[width=0.8\linewidth]{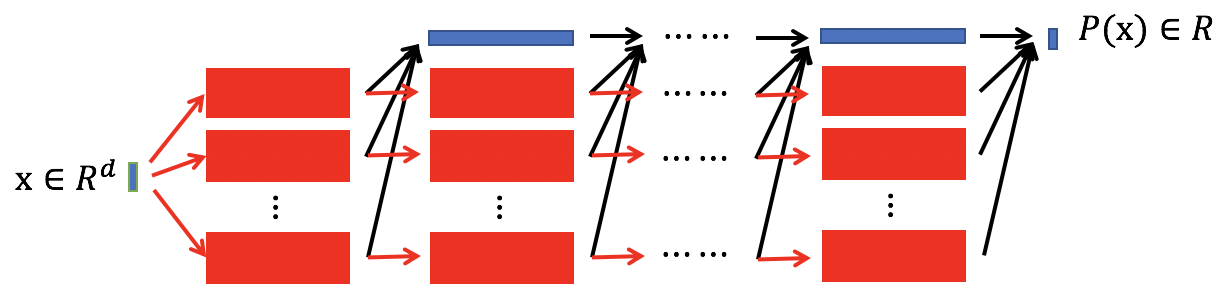}
		\caption{An illustration of the proof of Lemma \ref{lem:NNex2} (vi). Green vectors represent the input and output of the $\sigma_2$-$\tn{NN}$ $\tilde{\phi}(\xx)$ carrying out $P(\xx)$. Each red rectangle represents one $\sigma_2$-$\tn{NN}$ $\phi_j(\xx)$ and each blue rectangle represents one $\sigma_1$-$\tn{NN}$ of width $2$ as an identity map of $\R$. There are $ab\geq J$ red blocks with $a$ rows and $b$ columns. When $ab\geq J$, these sub-NN blocks can carry out all monomials $\xx^{\xal_j}$. In each column, the results of the multiplications of $\xx^{\xal_j}$ are added up to (indicated by black arrows) the input of the narrow $\sigma_1$-$\tn{NN}$, which can carry the sum over to the next column. Each red arrow passes $\xx$ to the next red block. After the calculation of $b$ columns, $J$ additions of the monomials $\xx^{\xal_j}$ have been implemented, resulting in the output $P(\xx)$.}
		\label{fig:NN2}
	\end{figure}

{Similar to the ReLU network case, we first present the following approximation result before showing our main theorem - Theorem \ref{corol:function_sigma_2}.}

\begin{theorem}\label{thm:function_sigma_2}
Suppose that $f \in C^s([0,1]^d)$ with $s\in \mathbb{N}^+$ satisfies $\| \partial^{\bm{\alpha}} f  \|_{L^{\infty}([0,1]^d)} < 1$ for any $\bm{\alpha} \in \mathbb{N}^d$ with $|\xal| \leq s$. For any $N,L \in \N^+$ satisfying $(L-2-\log_2 N)N \geq s$, there exists a $\sigma_2$-$\tn{NN}$ $\phi$ with width $16s^{d+1}d(N+2)\log_2{({8}N)}$ and depth $10(L+2)\log_2(4L)$ such that {$\| \phi(x) \|_{\mathcal{W}^{n,\infty}([0,1])} \leq s^d$ and} 

{\[
\| f - \phi \|_{\mathcal{W}^{n,\infty}([0,1]^d \backslash \Omega{([0,1]^d,K,\delta)})} \leq  2(s+1)^d 8^{s-n}N^{-2(s-n)/d}L^{-2(s-n)/d}
\]}
where $n < s$ is a positive integer, $K = \lfloor N^{1/d} \rfloor^2 \lfloor L^{2/d} \rfloor$ and $0 \leq \delta \leq \frac{1}{3K}$.
\end{theorem}

\begin{proof}

Similar to the proof of Theorem \ref{thm:function_sigma_1}, we set $K = \lfloor {N^{1/d}} \rfloor^2 \lfloor L^{2/d} \rfloor$ and let {$\Omega{([0,1]^d,K,\delta)}$} partition $[0,1]^d$ into $K^d$ cubes $Q_{\bm{\beta}}$ for $\bm{\beta} \in \{0,1,\cdots, K-1\}^d$ such that {\[[0,1]^d = \Omega{([0,1]^d,K,\delta)} \bigcup \big( \cup_{\bm{\beta} \in \{0,1,\cdots,K-1\}^d}   Q_{\bm{\beta}} \big). \]} For each $\bm{\beta}=[\beta_1,\beta_2,\cdots,\beta_d]^T \in \{0,1,\cdots,K-1\}^d$, we {define}
\[
Q_{\bm{\beta}} = \Bigg\{ \bm{x}=[x_1,x_2,\cdots,x_d]^T : x_i \in \Big[ \frac{\beta_i}{K}, \frac{\beta_i+1}{K} -\delta \cdot 1_{\{ \beta_i \leq K-2 \}} \Big]~\text{for}~i=1,2,\dots, d \Bigg\}.
\]
By Proposition \ref{prop:step_sigma_1}, there exists a $\sigma_2$-$\tn{NN}$ $\psi$ with width $4N+3$ and depth $4L+5$ such that
\[
\psi(x) = k,\quad \text{if} ~ x\in \big[\frac{k}{K},\frac{k+1}{K}- \delta \cdot 1_{\{ k \leq K-2 \}}\big]~\text{for}~k=0,1,\dots,K-1.
\]
Then, for each $\bm{\beta} \in \{0,1,\cdots,K-1\}^d, \psi(x_i)=\theta_i$ if $\bm{x} \in Q_{\bm{\beta}}~\text{for}~i=1,2,\cdots,d$.

Define
\[
\bm{\psi{(x)}}:= [\psi(x_1),\psi(x_2),\cdots, \psi(x_d)]^T/K \quad \text{for any}~\bm{x} \in [0,1]^d,
\]
then
\[
\bm{\psi{(x)}} = \bm{\beta}/K\quad \text{if}~\bm{x} \in Q_{\bm{\beta}}\quad \text{for}~\bm{\beta} \in \{0,1,\cdots,K-1\}^d.
\]

 Now, we fix a $\bm{\beta}\in \{ 0,1,\cdots,K-1 \}^d$ throughout the proof. For any $\bm{x} \in Q_{\bm{\beta}}$, by Taylor's expansion there exists a $\xi_{\bm{x}} \in (0,1)$ such that

\[
{ f(\bm{x})} = \sum_{|\xal|\leq s-1} \frac{\partial^{\bm{\alpha}}f(\bm{\psi(\bm{x})})}{\bm{\alpha}!} \bm{h}^{\bm{\alpha}} + \sum_{|\xal|= s} \frac{\partial^{\bm{\alpha}}f(\bm{\psi(\bm{x})} + \xi_{\bm{x}}\bm{h}  )}{\bm{\alpha}!} \bm{h}^{\bm{\alpha}},
\]
where $\bm{h} = \bm{x} - \bm{\psi{(\bm{x})}}$.

 By Lemma \ref{lem:NNex2} (iv), there exists a $\sigma_2$-$\tn{NN}$ $\widetilde{\phi}$ with width $4$ and depth $1$ such that 
\[
\widetilde{\phi}(x,y)=xy
\]
for any $x,y \in (-3,3).$

Note that it is trivial to construct $\sigma_2$-$\tn{NN}s$ $P_{\bm{\alpha}}(\bm{x})$ for $\bm{x}^{\bm{\alpha}}$ when $|\xal| \leq 1$. Thus, for each $\bm{\alpha} \in \mathbb{N}^d$ with $|\xal| \leq s-1$ and for any $N,L \in \N^+$ satisfying $(L-2-\log_2 N)N\geq s$, by Lemma \ref{lem:NNex2} (vi) with $a=b=J=1$, there exists a $\sigma_2$-$\tn{NN}$ $P_{\bm{\alpha}}$ with width $4N+2d+2$ and depth $L$ such that
\begin{equation*}
P_{\bm{\alpha}}(\bm{x})=\bm{x}^{\bm{\alpha}}
\end{equation*}
for any $\bm{x} \in \mathbb{R}^d$.

For each $i=0,1,\cdots,K^d-1$, we define
\[
\bm{\eta}(i) = [\eta_1,\eta_2,\cdots,\eta_d]^T \in \{0,1,\dots,K-1 \}^d
\]
such that $\sum_{j=1}^{d}\eta_j K^{j-1}=i$. We will drop the input $i$ in $\bm{\eta}(i)$ later for simplicity. For each $\bm{\alpha} \in \mathbb{N}^d$ with $|\xal| \leq s-1$, define
\[
\xi_{\bm{\alpha},i} = \big(\partial^{\bm{\alpha}}f\big(\frac{\eta}{K}\big) +1 \big)/2.
\]
Note that $K^d = (\lfloor N^{1/d} \rfloor^2 \lfloor L^{2/d} \rfloor)^d \leq N^2L^2$ and $\xi_{\bm{\alpha},i} \in [0,1]$ for $i=0,1,\cdots,K^d-1$. By Proposition \ref{prop:point_sigma_1}, there exists a $\sigma_1$-$\tn{NN}$ $\widetilde{\phi}_{\bm{\alpha}}$, which is also a $\sigma_2$-$\tn{NN}$, of width $16s(N+1)\log_{2}(8N)$ and depth $5(L+2)\log_{2}(4L)$ such that 
\[
| \widetilde{\phi}_{\bm{\alpha}}(i) -  \xi_{\bm{\alpha},i}| \leq N^{-2s}L^{-2s}, \quad \text{for}~i=0,1,\cdots, K^d - 1 ~\text{and}~ |\xal| \leq s-1.
\]
Define
\[
\phi_{\bm{\alpha}}(\bm{x}):=2\widetilde{\phi}_{\bm{\alpha}} \Bigg( \sum_{j=1}^{d} x_j K^{j-1}  \Bigg) -1,\quad \text{for any}~\bm{x}=[x_1,x_2,\cdots,x_d]^d \in \mathbb{R}^d.
\]
For each $|\xal| \leq s-1$, we know that $\phi_{\bm{\alpha}}$ is also of width $16s(N+1)\log_{2}(8N)$ and depth $5(L+2)\log_{2}(4L)$.

Then for each $\bm{\eta} = [\eta_1,\eta_2,\cdots,\eta_d]^T \in \{0,1,\dots,K-1 \}^d$ corresponding to $i=\sum_{j=1}^{d}\eta_j K^{j-1}$, each $\bm{\alpha} \in \mathbb{N}^d$ with $|\xal| \leq s-1$, we have
\[
\big| \phi_{\bm{\alpha}}(\frac{\bm{\eta}}{K})  - \partial^{\bm{\alpha}}f (\frac{\bm{\eta}}{K}) \big| = \Bigg| 2\widetilde{\phi}_{\bm{\alpha}} \bigg( \sum_{j=1}^{d} x_j K^{j-1}  \bigg) -1  - (2\xi_{\bm{\alpha},i} -1) \Bigg| = 2|\widetilde{\phi}_{\bm{\alpha}}(i) - \xi_{\bm{\alpha},i}  | \leq 2N^{-2s}L^{-2s}.
\]

From $\bm{\psi{(x)}}=\frac{\bm{\beta}}{K}$ for $\bm{x} \in Q_{\bm{\beta}}$, it follows that
\begin{eqnarray*}\nonumber
\| \phi_{\bm{\alpha} }(\bm{\psi (\bm{x}) }) - \partial^{\bm{\alpha}}f(\bm{\psi (\bm{x}) }) \|_{\mathcal{W}^{n,\infty}(Q_{\bm{\beta}})}&=& \| \phi_{\bm{\alpha} }(\bm{\psi (\bm{x}) }) - \partial^{\bm{\alpha}}f(\bm{\psi (\bm{x}) }) \|_{{L}^{\infty}(Q_{\bm{\beta}})}\\\nonumber
&=&\big| \phi_{\bm{\alpha}}(\frac{\bm{\beta}}{K})  - \partial^{\bm{\alpha}}f (\frac{\bm{\beta}}{K}) \big| \\ 
&\leq& 2N^{-2s}L^{-2s} =:\mathcal{E}_3.
\end{eqnarray*}

Note that since $ \phi_{\bm{\alpha} }(\bm{\psi (\bm{x}) }) - \partial^{\bm{\alpha}}f(\bm{\psi (\bm{x}) })$ {for $\bm{x} \in Q_{\bm{\beta}}$ is constant, its weak derivative is zero, which has given the above inequality.}

Define
\begin{equation}\label{eqn:main_nn_sigma2}
\phi(\bm{x}) = \sum_{|\xal| \leq s-1} \widetilde{\phi} \Big( \frac{\phi_{\bm{\alpha} }(\bm{\psi (\bm{x}) })}{\bm{\alpha}!}, P_{\bm{\alpha}} (\bm{h})  \Big)
\end{equation}
for any $\bm{x} \in \mathbb{R}^d$.

{Let us now first estimate the overall approximation error in semi-norm. We denote the $n$-th order derivatives of $f$ by $\partial^{\bm{\gamma}}f$ with $|\bm{\gamma}|=n$. For any $\bm{x} \in Q_{\bm{\beta}}$, by Taylor's expansion there exists a $\xi_{{\bm{x}}}^{\bm{\gamma}} \in (0,1)$ such that
\begin{eqnarray*}\label{eqn:Taylor_error_jth_order}
 \partial^{\bm{\gamma}}f &=& \sum_{|\xal|\leq s-n-1} \frac{\partial^{\bm{\alpha}}\partial^{\bm{\gamma}}f(\bm{\psi(\bm{x})})}{\bm{\alpha}!} \bm{h}^{\bm{\alpha}} + \sum_{|\xal|= s-n} \frac{\partial^{\bm{\alpha}}\partial^{\bm{\gamma}}f(\bm{\psi(\bm{x})} + \xi_{{\bm{x}}}^{\bm{\gamma}} \bm{h}  )}{\bm{\alpha}!} \bm{h}^{\bm{\alpha}}\\\nonumber
 &=& \partial^{\bm{\gamma}} \Big( \sum_{|\xal|\leq s-1} \frac{\partial^{\bm{\alpha}}f(\bm{\psi(\bm{x})})}{\bm{\alpha}!} \bm{h}^{\bm{\alpha}} \Big)+ \sum_{|\xal|= s-n} \frac{\partial^{\bm{\alpha}}\partial^{\bm{\gamma}}f(\bm{\psi(\bm{x})} + \xi_{\bm{x}}^{\bm{\gamma}} \bm{h}  )}{\bm{\alpha}!} \bm{h}^{\bm{\alpha}},
\end{eqnarray*}
where $\bm{h} = \bm{x} - \bm{\psi{(\bm{x})}}$.
\\
Therefore, for any $\bm{x} \in Q_{\bm{\beta}}$ we have
\begin{eqnarray*}\nonumber
&& | \phi(\bm{x}) - f(\bm{x}) |_{\mathcal{W}^{n,\infty}(Q_{\bm{\theta}})}\\
&\leq &\Big| \phi(\bm{x})- \sum_{|\xal|\leq s-1} \frac{\partial^{\bm{\alpha}}f(\bm{\psi(\bm{x})})}{\bm{\alpha}!} \bm{h}^{\bm{\alpha}}\Big|_{\mathcal{W}^{n,\infty}(Q_{\bm{\theta}})}  +  \Big| \sum_{|\xal|\leq s-1} \frac{\partial^{\bm{\alpha}}f(\bm{\psi(\bm{x})})}{\bm{\alpha}!} \bm{h}^{\bm{\alpha}}  - f(\bm{x}) \Big|_{\mathcal{W}^{n,\infty}(Q_{\bm{\theta}})}\\
&\leq &  {\sum_{|\xal| \leq s-1}  \Big| \widetilde{\phi} \Big( \frac{\phi_{\bm{\alpha} }(\bm{\psi (\bm{x}) })}{\bm{\alpha}!}, P_{\bm{\alpha}} (\bm{h})  \Big)  -  \frac{\partial^{\bm{\alpha}}f(\bm{\psi(\bm{x})})}{\bm{\alpha}!} \bm{h}^{\bm{\alpha}} \Big|_{\mathcal{W}^{n,\infty}(Q_{\bm{\beta}})}}\\
 &&+ \max_{|\bm{\gamma}|=n} \sum_{|\xal|= s-n} \Big\|   \frac{\partial^{\bm{\alpha}}\partial^{\bm{\gamma}}f(\bm{\psi(\bm{x})} + \xi_{\bm{x}}^{\bm{\gamma}}\bm{h}  )}{\bm{\alpha}!} \bm{h}^{\bm{\alpha}} \Big\|_{\mathcal{L}^{\infty}(Q_{\bm{\beta}})}\\
&\leq &  \underbrace{\sum_{|\xal| \leq s-1}  \Big\| \widetilde{\phi} \Big( \frac{\phi_{\bm{\alpha} }(\bm{\psi (\bm{x}) })}{\bm{\alpha}!}, P_{\bm{\alpha}} (\bm{h})  \Big)  -  \frac{\partial^{\bm{\alpha}}f(\bm{\psi(\bm{x})})}{\bm{\alpha}!} \bm{h}^{\bm{\alpha}} \Big\|_{\mathcal{W}^{n,\infty}(Q_{\bm{\beta}})}}_{=:E_1} \\\nonumber
&&+\underbrace{ \max_{|\bm{\gamma}|=n} \sum_{|\xal|= s-n} \Big\|   \frac{\partial^{\bm{\alpha}}\partial^{\bm{\gamma}}f(\bm{\psi(\bm{x})} + {\xi}^{(\bm{\gamma})}_{\bm{x}}\bm{h}  )}{\bm{\alpha}!} \bm{h}^{\bm{\alpha}} \Big\|_{\mathcal{L}^{\infty}(Q_{\bm{\beta}})}}_{=:E_2},
\end{eqnarray*}
where $E_2$ with ${\xi}^{(\bm{\gamma})}_{\bm{x}} \in (0,1)$ is the remainder resulted from Taylor's expansion of $\partial^{\bm{\gamma}}f$.
}

\begin{eqnarray}\nonumber
E_1 &=& \sum_{|\xal| \leq s-1}  \Big\| \widetilde{\phi} \Big( \frac{\phi_{\bm{\alpha} }(\bm{\psi (\bm{x}) })}{\bm{\alpha}!}, P_{\bm{\alpha}} (\bm{h})  \Big)  -  \frac{\partial^{\bm{\alpha}}f(\bm{\psi(\bm{x})})}{\bm{\alpha}!} \bm{h}^{\bm{\alpha}} \Big\|_{\mathcal{W}^{n,\infty}(Q_{\bm{\beta}})}\\\nonumber
&\leq& \underbrace{\sum_{|\xal| \leq s-1}  \Big\| \widetilde{\phi} \Big( \frac{\phi_{\bm{\alpha} }(\bm{\psi (\bm{x}) })}{\bm{\alpha}!}, P_{\bm{\alpha}} (\bm{h})  \Big)  -  \widetilde{\phi} \Big( \frac{\partial^{\bm{\alpha}}f(\bm{\psi(\bm{x})})  }{\bm{\alpha}!}, P_{\bm{\alpha}} (\bm{h})  \Big) \Big\|_{\mathcal{W}^{n,\infty}(Q_{\bm{\beta}})}}_{=:E_{1,1}} \\\nonumber
&&+ \underbrace{\sum_{|\xal| \leq s-1}  \Big\| \widetilde{\phi} \Big( \frac{\partial^{\bm{\alpha}}f(\bm{\psi(\bm{x})})  }{\bm{\alpha}!}, P_{\bm{\alpha}} (\bm{h})  \Big)  -  \frac{\partial^{\bm{\alpha}}f(\bm{\psi(\bm{x})})}{\bm{\alpha}!} \bm{h}^{\bm{\alpha}} \Big\|_{\mathcal{W}^{n,\infty}(Q_{\bm{\beta}})}}_{=:E_{1,2}}.
\end{eqnarray}

Hence, we can now measure $E_{1,1}$, $E_{1,2}$, and $E_2$:

\begin{eqnarray}\nonumber
E_{1,1} &=& \sum_{|\xal| \leq s-1}  \Big\| \widetilde{\phi} \Big( \frac{\phi_{\bm{\alpha} }(\bm{\psi (\bm{x}) })}{\bm{\alpha}!}, P_{\bm{\alpha}} (\bm{h})  \Big)  -  \widetilde{\phi} \Big( \frac{\partial^{\bm{\alpha}}f(\bm{\psi(\bm{x})})  }{\bm{\alpha}!}, P_{\bm{\alpha}} (\bm{h})  \Big) \Big\|_{\mathcal{W}^{n,\infty}(Q_{\bm{\beta}})}\\\nonumber
&\leq& \sum_{|\xal| \leq s-1}  \Bigg(   \underbrace{\Big\| \widetilde{\phi} \Big( \frac{\phi_{\bm{\alpha} }(\bm{\psi (\bm{x}) })}{\bm{\alpha}!}, P_{\bm{\alpha}} (\bm{h})  \Big)  -    \frac{\phi_{\bm{\alpha} }(\bm{\psi (\bm{x}) })}{\bm{\alpha}!} P_{\bm{\alpha}} (\bm{h})   \Big\|_{\mathcal{W}^{n,\infty}(Q_{\bm{\beta}})}}_{=0}  \\\nonumber
&&+ \underbrace{ \Big\| \widetilde{\phi} \Big( \frac{\partial^{\bm{\alpha}}f(\bm{\psi(\bm{x})})  }{\bm{\alpha}!}, P_{\bm{\alpha}} (\bm{h})  \Big)  -   \frac{\partial^{\bm{\alpha}}f(\bm{\psi(\bm{x})})  }{\bm{\alpha}!} P_{\bm{\alpha}} (\bm{h})   \Big\|_{\mathcal{W}^{n,\infty}(Q_{\bm{\beta}})}}_{=0}  \\\nonumber
&&+  \underbrace{\Big\|  \frac{\phi_{\bm{\alpha} }(\bm{\psi (\bm{x}) })}{\bm{\alpha}!} P_{\bm{\alpha}} (\bm{h})    -   \frac{\partial^{\bm{\alpha}}f(\bm{\psi(\bm{x})})  }{\bm{\alpha}!} P_{\bm{\alpha}} (\bm{h})   \Big\|_{\mathcal{W}^{n,\infty}(Q_{\bm{\beta}})}}_{\leq \mathcal{E}_3}  \Bigg)\\\nonumber
&\leq& \sum_{|\xal| \leq s-1} \mathcal{E}_3\\\nonumber
&\leq& s^d\mathcal{E}_3.
\end{eqnarray}

Note that the last inequality is followed by the fact that $\sum_{|\xal|\leq s-1} 1 = \sum_{i=1}^{s-1}(i+1)^{d-1} \leq s^d$. Similarly,
 \begin{eqnarray}\nonumber
E_{1,2} &=& \sum_{|\xal| \leq s-1}  \Big\| \widetilde{\phi} \Big( \frac{\partial^{\bm{\alpha}}f(\bm{\psi(\bm{x})})  }{\bm{\alpha}!}, P_{\bm{\alpha}} (\bm{h})  \Big)  -  \frac{\partial^{\bm{\alpha}}f(\bm{\psi(\bm{x})})}{\bm{\alpha}!} \bm{h}^{\bm{\alpha}} \Big\|_{\mathcal{W}^{n,\infty}(Q_{\bm{\beta}})}\\\nonumber
&\leq& \sum_{|\xal| \leq s-1}  \underbrace{\Big\| \widetilde{\phi} \Big( \frac{\partial^{\bm{\alpha}}f(\bm{\psi(\bm{x})})  }{\bm{\alpha}!}, P_{\bm{\alpha}} (\bm{h})  \Big)  -  \frac{\partial^{\bm{\alpha}}f(\bm{\psi(\bm{x})})}{\bm{\alpha}!} P_{\bm{\alpha}} (\bm{h}) \Big\|_{\mathcal{W}^{n,\infty}(Q_{\bm{\beta}})}}_{=0} \\\nonumber
&&+ \sum_{|\xal| \leq s-1}  {\underbrace{\Big\| \frac{\partial^{\bm{\alpha}}f(\bm{\psi(\bm{x})})}{\bm{\alpha}!} P_{\bm{\alpha}} (\bm{h}) -  \frac{\partial^{\bm{\alpha}}f(\bm{\psi(\bm{x})})}{\bm{\alpha}!} \bm{h}^{\bm{\alpha}} \Big\|_{\mathcal{W}^{n,\infty}(Q_{\bm{\beta}})}}_{_{=0}}}\\\nonumber
&\leq& 0.
\end{eqnarray}

{
We now provide a measure on $E_2$. First, by the assumption that $\| \partial^{\bm{\alpha}} f  \|_{L^{\infty}([0,1]^d)} < 1$ for any $\bm{\alpha} \in \mathbb{N}^d$ with $|\xal| \leq s$, we have
 \begin{eqnarray*}
E_2 &=&\max_{|\bm{\gamma}|=n}  \sum_{|\xal|= s-n} \Big\|   \frac{\partial^{\bm{\alpha}}\partial^{\bm{\gamma}}f(\bm{\psi(\bm{x})} + {\xi}^{(\bm{\gamma})}_{\bm{x}}   \bm{h}  )}{\bm{\alpha}!} \bm{h}^{\bm{\alpha}} \Big\|_{\mathcal{L}^{\infty}(Q_{\bm{\beta}})}\\
 &\leq &  \sum_{|\xal|= s-n} \Big\| \frac{1}{\bm{\alpha}!}  \bm{h}^{\bm{\alpha}} \Big\|_{{L}^{\infty}(Q_{\bm{\beta}})}\\
 &\leq & (s-n+1)^{d-1}K^{-(s-n)},
 \end{eqnarray*}
where the last inequality is followed by $\sum_{|\xal|=s-n} 1 \leq (s-n+1)^{d-1}$.
\\
Thus, we have shown that
 \begin{eqnarray*}
&&| \phi(\bm{x}) - f(\bm{x}) |_{\mathcal{W}^{n,\infty}(Q_{\bm{\beta}})} \\\nonumber
&\leq& E_{1,1} + E_{1,2} + E_2 \\\nonumber
&\leq& E_{1,1} + E_{1,2}   + (s-n+1)^{d-1} K^{-(s-n)}.
\end{eqnarray*}
We now estimate $| \phi(\bm{x}) - f(\bm{x}) |_{\mathcal{W}^{j,\infty}(Q_{\bm{\beta}})}, j=0,\dots,n-1$. Using similar arguments and considering Taylor's expansion of ${\partial^{\bm{\gamma}}}f$, $|\bm{\gamma}|\leq n-1$, with ${\xi}^{(\bm{\gamma})}_{\bm{x}} \in (0,1)$ in the remainder term, we have 
 \begin{eqnarray*}
&&| \phi(\bm{x}) - f(\bm{x}) |_{\mathcal{W}^{j,\infty}(Q_{\bm{\beta}})} \\
&\leq&\Big| \phi(\bm{x})- \sum_{|\xal|\leq s-1} \frac{\partial^{\bm{\alpha}}f(\bm{\psi(\bm{x})})}{\bm{\alpha}!} \bm{h}^{\bm{\alpha}} \Big|_{\mathcal{W}^{j,\infty}(Q_{\bm{\theta}})} +  \Big|\sum_{|\xal|\leq s-1} \frac{\partial^{\bm{\alpha}}f(\bm{\psi(\bm{x})})}{\bm{\alpha}!} \bm{h}^{\bm{\alpha}}  - f(\bm{x}) \Big|_{\mathcal{W}^{j,\infty}(Q_{\bm{\theta}})}\\
&\leq &  {\sum_{|\xal| \leq s-1}  \Big| \widetilde{\phi} \Big( \frac{\phi_{\bm{\alpha} }(\bm{\psi (\bm{x}) })}{\bm{\alpha}!}, P_{\bm{\alpha}} (\bm{h})  \Big)  -  \frac{\partial^{\bm{\alpha}}f(\bm{\psi(\bm{x})})}{\bm{\alpha}!} \bm{h}^{\bm{\alpha}} \Big|_{\mathcal{W}^{j,\infty}(Q_{\bm{\beta}})}}\\
 &&+ \max_{|\bm{\gamma}|=j} \sum_{|\xal|= s-j} \Big\|   \frac{\partial^{\bm{\alpha}}\partial^{\bm{\gamma}}f(\bm{\psi(\bm{x})} + {\xi}^{(\bm{\gamma})}_{\bm{x}}\bm{h}  )}{\bm{\alpha}!} \bm{h}^{\bm{\alpha}} \Big\|_{\mathcal{L}^{\infty}(Q_{\bm{\beta}})}\\
&\leq& {\sum_{|\xal| \leq s-1}  \Big\| \widetilde{\phi} \Big( \frac{\phi_{\bm{\alpha} }(\bm{\psi (\bm{x}) })}{\bm{\alpha}!}, P_{\bm{\alpha}} (\bm{h})  \Big)  -  \frac{\partial^{\bm{\alpha}}f(\bm{\psi(\bm{x})})}{\bm{\alpha}!} \bm{h}^{\bm{\alpha}} \Big\|_{\mathcal{W}^{n,\infty}(Q_{\bm{\beta}})}}\\
&& + (s-j+1)^{d-1} K^{-(s-j)}\\
&\leq& E_{1,1} + E_{1,2}   + (s-j+1)^{d-1} K^{-(s-j)},
\end{eqnarray*}
where the second last inequality is due to $\sum_{|\xal|=s-j} 1 \leq (s-j+1)^{d-1}$.
}

{
Using $K=\lfloor N^{1/d} \rfloor^2 \lfloor L^{2/d} \rfloor \geq \frac{N^{2/d} L^{2/d}}{8}$, we have
\begin{eqnarray}\nonumber
&& \| \phi(\bm{x}) - f(\bm{x}) \|_{\mathcal{W}^{n,\infty}(Q_{\bm{\beta}})} \\\nonumber
&=& \max{\big\{ \| \phi(\bm{x}) - f(\bm{x}) \|_{{L}^{\infty}(Q_{\bm{\beta}})} , \dots, | \phi(\bm{x}) - f(\bm{x}) |_{\mathcal{W}^{n-1,\infty}(Q_{\bm{\beta}})},  | \phi(\bm{x}) - f(\bm{x}) |_{\mathcal{W}^{n,\infty}(Q_{\bm{\beta}})} \big\}}  \\\nonumber
&\leq& E_{1,1} + E_{1,2} + (s+1)^{d-1} K^{-(s-n)} \\\nonumber
&\leq& s^d\mathcal{E}_3  + (s+1)^{d-1} K^{-(s-n)}\\\nonumber
&\leq& (s+1)^d( K^{-(s-n)} +\mathcal{E}_3 )\\\nonumber
&\leq& (s+1)^d\big(  K^{-(s-n)} + 2N^{-2s}L^{-2s} \big)\\\nonumber
&\leq&(s+1)^d\big(  8^{s-n}N^{-2(s-n)/d}L^{-2(s-n)/d} +2N^{-2s}L^{-2s} \big)\\\nonumber
&\leq&(s+1)^d  (8^{s-n}+2)N^{-2(s-n)/d}L^{-2(s-n)/d}  \\\nonumber
&\leq& 2(s+1)^d 8^{s-n}N^{-2(s-n)/d}L^{-2(s-n)/d}.
\end{eqnarray}
}

Since $\bm{\beta} \in \{0,1,2,\cdots,K-1  \}^d$ is arbitrary and the fact that {$[0,1]^d \backslash \Omega{([0,1]^d,K,\delta)} \subseteq  \cup_{\bm{\beta} \in \{0,1,\cdots,K-1\}^d}  Q_{\bm{\beta}} $}, we have
{\[
\| \phi(\bm{x}) - f(\bm{x}) \|_{\mathcal{W}^{n,\infty}([0,1]^d \backslash \Omega{([0,1]^d,K,\delta)})} \leq  2(s+1)^d 8^{s-n}N^{-2(s-n)/d}L^{-2(s-n)/d}.
\]}

{
Furthermore, we have
\begin{eqnarray*}
\| \phi(\bm{x})\|_{\mathcal{W}^{1,\infty}([0,1]^d )} &=& \Bigg\|   \sum_{|\xal| \leq s-1} \widetilde{\phi} \Big( \frac{\phi_{\bm{\alpha} }(\bm{\psi (\bm{x}) })}{\bm{\alpha}!}, P_{\bm{\alpha}} (\bm{h})  \Big) \Bigg\|_{\mathcal{W}^{1,\infty}([0,1]^d )}\\
&\leq &  \sum_{|\xal| \leq s-1}  \| \widetilde{\phi}\|_{\mathcal{W}^{1,\infty}([0,1]^d)}\\
&\leq& s^d.
\end{eqnarray*}
}

As last, we finish the proof by estimating the width and depth of the network implementing $\phi(\bm{x})$. From (\ref{eqn:main_nn_sigma2}), assuming for any $N,L \in \N^+$ satisfying $(L-2-\log_2 N)N \geq s$, we know that $\phi(\bm{x})$ consists of the following subnetworks:

\begin{enumerate}
\item $\bm{\psi} \in \tn{NN}(\textrm{width} \leq d(4N+3);\textrm{depth} \leq 4L+5)$;
\item $ \phi_{\bm{\alpha} } \in \tn{NN}(\textrm{width} \leq 16s(N+1)\log_{2}(8N);\textrm{depth} \leq 5(L+2)\log_{2}(4L))$;
\item $P_{\bm{\alpha}} \in \tn{NN}(\textrm{width} \leq 4N+2d+2 ;\textrm{depth} \leq L)$;
\item $\widetilde{\phi} \in \tn{NN}(\textrm{width} \leq 4;\textrm{depth} \leq 1)$.
\end{enumerate}
Thus, we can infer that the $\phi$ can be implemented by a ReLU network with width $16s^{d+1}{ d}(N+2)\log_2{(8N)}$ and depth 
\begin{equation*}
	\begin{split}
		(4L+5)+1+L+5(L+2)\log_2(4L) +3\le
		10(L+2)\log_2(4L)
	\end{split}.
\end{equation*} Hence, we have finished the proof.
\qed
\end{proof}

{We can now prove Theorem \ref{corol:function_sigma_2} using Theorem \ref{thm:function_sigma_2}.}

{
\begin{proof}[Proof of Theorem \ref{corol:function_sigma_2}]
When $f$ is a constant function, the statement is trivial.
By Theorem \ref{thm:function_sigma_2}, there exists a $\sigma_2$-$\tn{NN}$ $\phi$ with width $16s^{d+1}d(N+2)\log_2{({8}N)}$ and depth $10(L+2)\log_2(4L)$ such that $\|{\phi}\|_{\mathcal{W}^{n,\infty}([0,1]^d)} \leq s^d$ and
\[
\| f - \phi \|_{\mathcal{W}^{n,p}([0,1]^d)} \leq  2(s+1)^d 8^{s-n}N^{-2(s-n)/d}L^{-2(s-n)/d},
\]
Now, we set $K=\lfloor N^{1/d} \rfloor^2 \lfloor L^{2/d}\rfloor$ and choose a small $\delta$ such that
\[
Kd\delta \leq \big(N^{-2(s-n)/d}L^{-2(s-n)/d}\big)^p.
\]
Then, we have
{
\begin{align*}
\| f - \phi \|^p_{\mathcal{W}^{n,p}([0,1]^d)} &= \| f - \phi \|^p_{\mathcal{W}^{n,p}(\Omega{([0,1]^d,K,\delta)})} +   \| f - \phi \|^p_{\mathcal{W}^{n,p}([0,1]^d \backslash \Omega{([0,1]^d,K,\delta)})}\\
&\leq  Kd\delta(s^d)^p +  \big( 2(s+1)^d 8^{s-n}N^{-2(s-n)/d}L^{-2(s-n)/d} \big)^p\\
&\leq  \big(s^d N^{-2(s-1)/d}L^{-2(s-1)/d}\big)^p +  \big(2(s+1)^d 8^{s-n}N^{-2(s-n)/d}L^{-2(s-n)/d}\big)^p\\
&\leq  (3(s+1)^d 8^{s-n}N^{-2(s-n)/d}L^{-2(s-n)/d})^p.
\end{align*}
}
Hence, we have
\[
\| f - \phi \|_{\mathcal{W}^{1,p}([0,1]^d)} \leq 3(s+1)^d 8^{s-n}N^{-2(s-n)/d}L^{-2(s-n)/d}.
\]
\qed
\end{proof}
}

\section{Conclusions}
We have given a number of theoretical results on explicit error characterization for approximating smooth functions using deep ReLU networks and their smoother variants. Our results measured in Sobolev norms are well-suited for studying solving high-dimensional PDEs. Further generalizing our analysis to {other function spaces including the H\"{o}lder space} and other neutral networks such as the Floor-ReLU networks will be an interesting direction for research work. Numerical investigation of our findings in the setting for solving parametric PDEs will also be left as future work.

\section*{Acknowledgements}
{The authors would like to thank the editor and the two anonymous reviewers for their valuable comments and constructive suggestions, which have greatly improved this paper.} S. Hon was partially supported by the Hong Kong RGC under Grant 22300921, a start-up grant from the Croucher Foundation, and a Tier 2 Start-up Grant from the Hong Kong Baptist University. H.~Yang was partially supported by the US National Science Foundation under award DMS-1945029.




\bibliographystyle{plain}

\end{document}